\documentclass[a4paper]{article}
\usepackage{hyperref}
\usepackage{graphicx}
\usepackage{mathrsfs}
\usepackage{amsmath}
\usepackage{enumerate} 
\usepackage{amsxtra,amssymb,latexsym, amscd,amsthm}
\usepackage{indentfirst}
\usepackage{color}
\usepackage[utf8]{inputenc}
\usepackage[mathscr]{eucal}
\usepackage{amsfonts}
\usepackage{graphics}
\usepackage{multirow}
\usepackage{array}
\usepackage{subfigure}

\usepackage{cite}
\usepackage{wrapfig}

\usepackage{algorithm}
\usepackage[noend]{algpseudocode}

\makeatletter
\def\algbackskip{\hskip-\ALG@thistlm}
\makeatother

\usepackage{ushort}
\usepackage{pdflscape}
\usepackage{rotating}
\usepackage{pgfplots}

\usepackage{subfigure}
\usepackage{tikz}
\usepackage{adjustbox}

\newcommand{\stirling}[2]{\genfrac{\{}{\}}{0pt}{}{#1}{#2}}
\newcommand{\stirlingii}[3][1]{\scalebox{#1}{$\stirling{#2}{#3}$}}

\usepackage{array}
\usepackage{booktabs}
\newcommand{\footremember}[2]{%
    \footnote{#2}
    \newcounter{#1}
    \setcounter{#1}{\value{footnote}}%
}
\newcommand{\footrecall}[1]{%
    \footnotemark[\value{#1}]%
} 

\def\x{\mathbf{x}}

\def\a{\mathbf{a}}

\def\R{{\mathbb R}}

\def\N{{\mathbb N}}

\DeclareMathOperator{\conv}{conv}
\DeclareMathOperator{\rank}{rank }

\DeclareMathOperator{\eig}{eig}

\DeclareMathOperator{\diag}{diag}

\DeclareMathOperator{\trace}{trace}

\newtheorem{lemma}{\bf Lemma}[section]

\newtheorem{assumption}{\bf Assumption}[section]
\newtheorem{example}{\bf Example}[section]
\newtheorem{theorem}{\bf Theorem}[section]
\newtheorem{proposition}{\bf Proposition}[section]
\newtheorem{corollary}{\bf Corollary}[section]
\newtheorem{definition}{\bf Definition}[section]
\newtheorem{remark}{\bf Remark}[section]
\providecommand{\keywords}[1]
{
  \small	
  \textbf{\textbf{Keywords:}} #1
}
\pgfplotsset{compat=1.13}
\begin{document}

\small
\definecolor{qqzzff}{rgb}{0,0.6,1}
\definecolor{ududff}{rgb}{0.30196078431372547,0.30196078431372547,1}
\definecolor{xdxdff}{rgb}{0.49019607843137253,0.49019607843137253,1}
\definecolor{ffzzqq}{rgb}{1,0.6,0}
\definecolor{qqzzqq}{rgb}{0,0.6,0}
\definecolor{ffqqqq}{rgb}{1,0,0}
\definecolor{uuuuuu}{rgb}{0.26666666666666666,0.26666666666666666,0.26666666666666666}
\newcommand{\vi}[1]{\textcolor{blue}{#1}}
\newif\ifcomment
\commentfalse
\commenttrue
\newcommand{\comment}[3]{%
\ifcomment%
	{\color{#1}\bfseries\sffamily#3%
	}%
	\marginpar{\textcolor{#1}{\hspace{3em}\bfseries\sffamily #2}}%
	\else%
	\fi%
}
\newcommand{\victor}[1]{
	\comment{blue}{V}{#1}
}
\definecolor{oucrimsonred}{rgb}{0.6, 0.0, 0.0}
\newcommand{\jean}[1]{
	\comment{oucrimsonred}{J}{#1}
}
\definecolor{cadmiumgreen}{rgb}{0.0, 0.42, 0.24}
\newcommand{\hoang}[1]{
	\comment{cadmiumgreen}{H}{#1}
}
\title{A hierarchy of spectral relaxations
for polynomial optimization}
\author{%
Ngoc Hoang Anh Mai\footremember{1}{CNRS; LAAS; 7 avenue du Colonel Roche, F-31400 Toulouse; France.} \and %
   Jean-Bernard Lasserre\footrecall{1} \footremember{2}{Universit\'e de Toulouse; LAAS; F-31400 Toulouse, France.} \and  %
   Victor Magron\footrecall{1} %
  }
\maketitle
\begin{abstract}
We show that (i) any constrained polynomial optimization problem (POP) has an equivalent formulation on a variety contained in an Euclidean sphere and (ii) the resulting semidefinite relaxations in the moment-SOS hierarchy have the constant trace property (CTP) for the involved matrices. 
We then exploit the CTP to avoid solving the semidefinite relaxations via interior-point methods and rather use ad-hoc spectral methods that minimize the largest eigenvalue of a matrix pencil. Convergence to the optimal value of the semidefinite relaxation is guaranteed. 
As a result we obtain a hierarchy of nonsmooth  ``spectral relaxations'' of the initial POP.
Efficiency and robustness of this spectral hierarchy is tested against several equality constrained POPs on a sphere as well as on a sample
of randomly generated quadratically constrained quadratic problems (QCQPs).
\end{abstract}
\tableofcontents
\section{Introduction}

The moment-sums of squares (moment-SOS) hierarchy for solving polynomial optimization problems (POP) consists of solving a sequence of semidefinite programming (SDP) relaxations of increasing size. 
Thanks to powerful positivity certificates from real algebraic geometry, its associated monotone sequence of optimal values converges to the global optimum \cite{lasserre2001global}.  Even though this procedure is efficient, with generically  finite convergence \cite{nie2014optimality}, it suffers from two main drawbacks: \\

(i) In view of the current status of SDP solvers, it is limited to problems of modest size unless some sparsity and/or symmetry can be exploited.

(ii) When solving the semidefinite (SDP) relaxations of the hierarchy by interior-point methods (as do most current SDP solvers) the computational cost is quite high.\\

Recent efforts have tried to overcome these drawbacks: 

(a) By designing computationally cheaper hierarchies of convex relaxations based on alternative positivity certificates such as the bounded degree SOS hierarchy \cite{lasserre2017bounded}, nonnegative circuits relying on geometric programming \cite{deWolff17} or second-order cone programming  \cite{wang2019second}, and arithmetic-geometric-exponentials \cite{SAGE16}  relying on relative entropy programming.

(b) By  exploiting certain sparsity patterns in the POP formulation, based on correlative sparsity \cite{Waki06SparseSOS,Las06SparseSOS} or term sparsity \cite{wang2019tssos,wang2020chordal,cstssos}, possibly combined with (a). 

(c) By exploiting a \emph{Constant Trace Property} (CTP)  
of semidefinite relaxations associated with POPs coming from combinatorial optimization \cite{helmberg2000spectral,yurtsever2019scalable}. This permits to solve the semidefinite relaxation with ad-hoc method, like, e.g., limited-memory bundle methods, instead of the costly interior-point methods.

The present paper is part of the latter type-(c) efforts.

\subsection{Background on SDP with CTP}
One way to exploit the CTP of matrices in SDPs is
to consider the dual which reduces 
to minimize the maximum eigenvalue of a symmetric matrix pencil  \cite{helmberg2000spectral}.
For problems of moderate size one may solve the latter problem with interior-point methods \cite{ben2001lectures}. 
However for larger-scale instances, 
running a single iteration becomes computationally too demanding and therefore one has to 
use alternative methods, and in particular first-order methods.

To solve large-scale instances of this maximal eigenvalue minimization problem, 
two types of first-order methods can be used:   subgradient descent or variants of the mirror-prox algorithm \cite{nemirovsky1983problem}, and spectral bundle methods \cite{helmberg2000spectral}. 
In other methods of interest based on non-convex formulations  \cite{burer2003nonlinear,journee2008low}, the problem is directly  solved  over the set of low rank matrices. These latter approaches are particularly efficient for problems where the solution is low rank, e.g., for matrix completion or combinatorial relaxations.

Despite their empirical efficiency, the computational complexity of 
spectral bundle and low rank methods is still not completely understood. 
This is in contrast with methods based on stochastic smoothing results for which explicit 
computational complexity estimates are available. For instance in \cite{d2014stochastic} 
smooth stochastic approximations of the maximum eigenvalue function are obtained via
rank-one Gaussian perturbations.
In \cite{overton1995second} Newton's method is used, assuming that the multiplicity of 
the maximal eigenvalue is known in advance.

By combining quasi-Newton methods (e.g.  Broyden-Fletcher-Goldfarb-Shanno (BFGS) method or its so-called ``Limited-memory'' version (L-BFGS)   \cite{nocedal1980updating}) with adaptive gradient sampling \cite{burke2005robust,kiwiel2007convergence}, convergence guarantees are obtained for certain non smooth problems while keeping good empirical performance \cite{lewis2013nonsmooth,curtis2015quasi}.  

Another hybrid method is the Limited-Memory Bundle Method (LMBM) which combines L-BFGS with bundle methods \cite{haarala2007globally,haarala2004new}:
Briefly, L-BFGS is used in the line search procedure to determine the step sizes in the bundle method.  LMBM enjoys global convergence for locally Lipschitz continuous functions which are not necessarily differentiable.

Finally the more recent SketchyCGAL algorithm \cite{yurtsever2019scalable} 
also uses limited memory and arithmetic.
It combines a primal-dual optimization scheme together with a randomized sketch for  low-rank  matrix  approximation.   
Assuming that strong duality holds,
it provides a near-optimal low-rank approximation. 
A variant of SketchyCGAL can handle SDPs with \emph{bounded} (instead of \emph{constant}) trace property.

Concerning SDPs coming from relaxations in polynomial optimization,
Malick and Henrion \cite[Section 3.2.3]{henrion2012projection} have used the CTP
to provide an efficient algorithm for 
 unconstrained polynomial optimization problems. At last but not least,
the CTP trivially holds for Shor's relaxation \cite{shor1987quadratic} of 
combinatorial optimization problems
formulated as linear-quadratic POPs on the discrete hypercube $\{-1,1\}^n$.
This fact 
has been exploited in Helmberg and Rendl  \cite{helmberg2000spectral} to avoid solving the associated SDP via interior-point methods.

\subsection{Contribution}
A novelty with respect to previous (c)-efforts is to show that \emph{every} POP 
on a compact basic semialgebraic set has an equivalent 
equality constrained POP formulation on an Euclidean  sphere (possibly after adding some artificial variables) such that each of its semidefinite relaxations in the moment-SOS hierarchy has the CTP. 
We call CTP-POP such a formulation of POPs.
Therefore to solve each semidefinite relaxation
of a CTP-POP one may avoid the computationally costly interior-point methods in some cases. Indeed 
as the dual reduces to minimize the largest eigenvalue of a matrix pencil, one may rather use efficient ad-hoc non smooth methods as those invoked above. 








\subsection*{Main results}

    (I) In Section \ref{sec:POP.sphere}, we prove that each semidefinite moment relaxation indexed by $k\in\N$:
    \begin{equation}\label{eq:moment.hierarchy.intro}
-\tau_k = \sup _{\mathbf X\in \mathcal{S}_k} \{ \left< \mathbf C_k,\mathbf X\right>\,:\,\mathcal{A}_k \mathbf X=\mathbf b_k\,,\, \mathbf X \succeq 0\}\,,
\end{equation}
    of the moment-SOS hierarchy associated with an equality constrained POP on an Euclidean sphere of $\R^n$ has CTP (see Lemma \ref{lem:trace.constant.property}), i.e.,
    \[\forall \ \mathbf X\in \mathcal{S}_k\,,\,\mathcal{A}_k \mathbf X=\mathbf b_k\Rightarrow \trace(\mathbf X)=a_k\,,\]
    where $\mathcal A_k^T:\R^{m_k}\to \mathcal S_k$ is a linear operator with $\mathcal S_k$ being the set of  real symmetric matrices of size $\binom {n+k}{n}$, $\mathbf C_k\in \mathcal S_k$ and $\mathbf b_k\in\R^{m_k}$ with $m_k=\mathcal O\left(\binom {n+k}{n}^2\right)$. 
  Following the framework by Helmberg and Rendl \cite{helmberg2000spectral}, SDP \eqref{eq:moment.hierarchy.intro}
boils down to minimizing the largest eigenvalue of a matrix pencil:
\begin{equation}\label{eq:nonsmooth.hierarchy.intro}
\begin{array}{rl}
-\tau_k &= \inf \{ a_k\lambda_1(\mathbf C_k-\mathcal{A}_k^T\mathbf z)+\mathbf b_k^T\mathbf z \,:\,\mathbf z\in \R^{m_k}\}\,,
\end{array}
\end{equation}
where $\lambda_1(\mathbf A)$ stands for the largest eigenvalue of $\mathbf A$. 

Hence 
\eqref{eq:nonsmooth.hierarchy.intro} form what we call a hierarchy of (non smooth, convex) \emph{spectral relaxations} of the equality constrained POP on a sphere. 
Convergence of $(\tau_k)_{k\in\N}$ to the optimal value $f^\star$ of the initial POP 
is guaranteed with rate at least 
$\mathcal{O}(k^{-1/c})$ (see Theorem \ref{theo:conver.semi.hie}). 

In addition, existence of an optimal solution  of 
the spectral relaxation  \eqref{eq:nonsmooth.hierarchy.intro} is guaranteed for sufficiently large $k$ under certain conditions on the POP (see Proposition \ref{prop:strong.duality.attainability.POP.sphere}).
Finally, when the set of global minimizers 
of the equality constrained POP on the sphere
is finite, we also describe how to obtain an optimal solution $\x^\star$ 
via an optimal solution $\bar {\mathbf z}$ of  \eqref{eq:nonsmooth.hierarchy.intro}.

  (II) In Section \ref{sec:POP.on.compact.set} we prove that 
any POP on a compact basic semialgebraic set (including a ball constraint $R-\|\mathbf x\|_2^2\ge 0$) has an equivalent equality constrained POP (called CTP-POP) on a sphere of $\R^{n+l_g+1}$, where $l_g$ is the number of inequality constraints of the initial POP. 
This CTP-POP can be solved by using spectral relaxations \eqref{eq:nonsmooth.hierarchy.intro}.

(III)  We describe Algorithm \ref{alg:sol.nonsmooth.hier} to handle a given equality constrained POP on the sphere.
It consists of
handling each semidefinite relaxation \eqref{eq:moment.hierarchy.intro} by solving the spectral formulation  \eqref{eq:nonsmooth.hierarchy.intro}, 
with a nonsmooth optimization procedure
 chosen in advance by the user in 
our software library, called \href{https://github.com/maihoanganh/SpectralPOP}{SpectralPOP}. 
This library supports the three optimization subroutines 
LMBM  \cite{haarala2007globally,haarala2004new}, proximal bundle (PB) \cite{helmberg2000spectral}, and SketchyCGAL  \cite{yurtsever2019scalable}.
Our default method in  Algorithm \ref{alg:sol.nonsmooth.hier} is LMBM.

(IV)  Finally, efficiency and robustness of SpectralPOP are illustrated in Section \ref{sec:benchmark} on extensive  benchmarks. We solve several (randomly generated) dense equality constrained QCQPs on the unit sphere by running
Algorithm \ref{alg:sol.nonsmooth.hier} and compare results with those obtained with the standard moment-SOS hierarchy.
Suprisingly SpectralPOP can provide the optimal value as well as an optimal solution with high accuracy, and up to twenty five times faster than the semidefinite hierarchy. 
For instance, SpectralPOP can solve the first relaxation of 
minimization problem of dense quadratic polynomials on the unit sphere  with up to $n=500$ variables in about $47$ seconds and up to $1500$ variables in about $3500$ seconds on a standard laptop computer.
Eventually, an extended application of spectral relaxations for squared polynomial systems is presented in this section.
In view of numerical experiments, our strategy is currently well-suited to equality constrained problems  rather than POPs with several inequality constraints.

In \cite{helmberg2000spectral}, Helmberg and Rendl propose a spectral bundle method (based on Kiwiel's proximal bundle method \cite{kiwiel1990proximity}) to solve an SDP relying on the  maximal eigenvalue minimization problem of the form \eqref{eq:nonsmooth.hierarchy.intro}.
This method works better than interior-point algorithms for very large-scale SDPs, 
 when the number of trace equality constraints is  not larger than the size of the positive semidefinite matrix (e.g., Shor's relaxation of MAXCUT problems).
However this method is not always more efficient than interior-point solvers (e.g., SDPT3) 
 for instance when the SDPs involve 
 a number of trace equality constraints which is larger than the size of the positive semidefinite matrix, as reported in \cite[Table 1-6]{helmberg2014spectral}. 
 Unfortunately this latter type of SDP is the generic form of moment-SOS relaxations for POPs and thus is not suitable to be solved by Helmberg-Rendl's spectral bundle method. 
By contrast with previous works, our numerical results show that the combination between Helmberg-Rendl's spectral formulation and LMBM is cheaper and faster than Mosek (the currently fastest SDP solver based on interior-point method) while maintaining the same accuracy when solving moment relaxations of equality constrained POPs on a sphere.
\section{Background and Preliminary Results}
With $\x = (x_1,\dots,x_n)$, let $\R[\x]$ stands for the ring of real polynomials and let  $\Sigma[\x]\subset\R[\x]$ be its subset of SOS polynomials.
Let us note $\R[\x]_t$ and $\Sigma[\x]_t$ their respective restrictions 
to polynomials of degree at most $t$ and $2t$. 
Given $\alpha = (\alpha_1,\dots,\alpha_n) \in \N^n$, we note $| \alpha| := \alpha_1 + \dots + \alpha_n$.
Let $(\x^{ \alpha})_{\alpha\in\N^n}$ 
be the canonical basis of monomials for  $\R[\x]$ (ordered according to the graded lexicographic order) and 
$\mathbf v_t(\x)$ be the vector of monomials up to degree $t$, with length $\stirlingii n{t} := \binom{n+t}{n}$.
A polynomial $p\in\R[\x]_t$ is written as  
$p(\x)\,=\,\sum_{| \alpha | \leq t} p_\alpha\,\x^\alpha\,=\,\mathbf{p}^T\mathbf  v_t(\x)$, 
where $\mathbf{p}=(p_\alpha)\in\R^{\stirlingii n t}$ is its vector of coefficients in the canonical basis. 
The $l_1$-norm of a polynomial $p$ is given by the $l_1$-norm of its vector of coefficients $\mathbf{p}$, that is $\|\mathbf{p}\|_1 := \sum_{\alpha} |p_\alpha|$. Given $\a\in\R^n$, the $l_2$-norm of $\a$ is  $\|\a\|_2:=(a_1^2+\dots+a_n^2)^{1/2}$. For every $l\in\N^{>0}$, note $[l]:=\{1,\dots,l\}$ and $[0]:=\emptyset$.

\paragraph{Riesz linear functional.} Given  a real-valued sequence $\mathbf y=(y_\alpha)_{\alpha\in\N^n}$, define the 
Riesz linear functional $L_{\mathbf y}:\R[ \mathbf x ] \to \R$, $f\mapsto {L_{\mathbf y}}( f ) := \sum_{\alpha} f_\alpha y_\alpha$.  
A real infinite (resp. finite) sequence $( y_\alpha)_{\alpha  \in \N^n}$ (resp. $( y_\alpha)_{\alpha  \in \N^n_t}$) has a \emph{representing measure} if there exists a finite Borel measure $\mu$ such that $y_\alpha  = \int_{\R^n} {x^\alpha d\mu(\mathbf x)}$ is satisfied for every $\alpha  \in {\N^n}$ (resp. $\alpha  \in {\N^n_t}$). In this case, $( y_\alpha)_{\alpha  \in \N^n}$ is called be the moment sequence of $\mu$. 

\paragraph{Moment matrices.} The moment matrix of degree $d$ associated with a real-valued sequence $\mathbf y=(y_\alpha)_{\alpha  \in \N^n}$ and $d\in \N^{>0}$,
is the real symmetric matrix $\mathbf M_d(\mathbf y)$ of size $\stirlingii n d$,  with entries
$( y_{\alpha  + \beta })_{\alpha,\beta\in \N^n_d} $. 

\paragraph{Localizing matrices.} The localizing matrix of degree $d$ associated with $\mathbf y=(y_\alpha)_{\alpha  \in \N^n}$ and $p = \sum_{\gamma} p_\gamma x^\gamma  \in \R[\mathbf x]$, is the real symmetric matrix $\mathbf M_d(p\,\mathbf y)$ of size $\stirlingii n d$ 
with entries $(\sum_\gamma  {{p_\gamma }{y_{\gamma  + \alpha  + \beta }}})_{\alpha, \beta\in \N^n_d}$.

\subsection{General POPs on basic compact semialgebraic sets}
\label{sec:pop.on.sphere.ball}
A polynomial optimization problem is of the form
\begin{equation}\label{eq:POP.on.variety.ball}
f^\star:=\inf \{f(\mathbf x)\ :\ \mathbf x\in S(g,h)\}\,,
\end{equation}
where $S(g,h)$ is a basic semialgebraic set defined as follows:
\begin{equation}
\label{eq:variety-V.ball}
    S(g,h) \,:=\,\{\,\mathbf x\in\R^n:\: g_i(\mathbf x)\ge  0\,,\,i\in [l_g]\,;\,h_j(\mathbf x)\,=\,0\,,\:j\in [l_h]\,\}
\end{equation}
for some polynomials $f,g_i,h_j\in\R[\mathbf x]$. We note $g:=\{g_i\}_{i\in[l_g]}$ and $h:=\{h_j\}_{j\in[l_h]}$. 
For $p\in\R[\x]$, let $\lceil p\rceil:=\lceil{\rm deg}(p)/2\rceil$.

If $S(g,h)\ne \emptyset$ then $f^\star <\infty$ and POP \eqref{eq:POP.on.variety.ball} has at least one global minimizer.  Next, as we are concerned with POPs on compact feasible sets, we assume that $S(g,h)\subset B_R^n$, where $B_R^n:=\{\mathbf x\in\R^n: R-\Vert \mathbf x\Vert_2^2\geq 0\}$.
In addition, if $l_g\neq0$ then we may and will assume that
$g_1:=R-\|\mathbf x\|_2^2$.

\paragraph{Second-order sufficient condition.}
Given $(\lambda_i)_{i\in[l_g]}$ and
$(\gamma_j)_{i\in[l_h]}$, let:
\[\x\mapsto \mathcal{L}(\x,\lambda,\gamma)\,:=\,
f(\x)-\sum_{i\in[l_g]}\lambda_i\,g_i(\x)
-\sum_{j\in[l_h]}\gamma_j\,h_j(\x),\quad \x\in\R^n.\]
Given $\x\in S(g,h)$, let $J(\x):=\{\,i\in[l_g]: g_i(\x)=0\,\}$.
\begin{definition}
\label{def-S2}
The second-order sufficient condition
(S2) holds at $\x^\star\in S(g,h)$ under the three following conditions.
\begin{itemize}
\item {\bf Constraint qualification:} The family $\{\nabla g_i(\x^\star),\nabla h_j(\x^\star)\}_{i\in J(\x^\star), j\in [l_h]}$ is linearly independent. This implies the existence of 
KKT-Lagrange multipliers $\lambda^\star_i\ge0$, 
$i\in [l_g]$, and $\gamma_j\in\R$, $j\in [l_h]$, such that
$\nabla\mathcal{L}(\x^\star,\lambda^\star,\gamma^\star)=0$ and $\lambda^\star_i\,g_i(\x^\star)=0$ for all
$i\in[l_g]$.
\item {\bf Strict complementarity:} $\lambda_i^\star +g_i(\x^\star)>0$, for all $i\in [l_g]$.
\item $\mathbf{u}^T\nabla^2\mathcal{L}(\x^\star,\lambda^\star,\gamma^\star)\,\mathbf{u}>0$ for all $\mathbf{u}\neq0$ such that 
$\mathbf{u}^T\nabla \mathcal{L}(\x^\star,\lambda^\star,\gamma^\star)=0$.
\end{itemize}
\end{definition}

\paragraph{The Moment-SOS hierarchy.}
Given $k\in\N$, the set
\[Q(g,h): = \left\{ \sigma_0+\sum_{i = 1}^{l_g}\sigma_ig_i + \sum_{j = 1}^{l_h} {\psi _jh_j} \ :\ \sigma_0 \in \Sigma[ \mathbf x]\,,\, \sigma_i \in \Sigma[ \mathbf x]\,,\, \psi_j\in \R[\mathbf x]\right\}\,.\]
is the \emph{quadratic module} associated with the  semialgebraic set $S(g,h)$, while the set
\[Q_k(g,h): = \left\{ \sigma_0+\sum_{i = 1}^{l_g}\sigma_ig_i + \sum_{j = 1}^{l_h} {\psi _jh_j} \left|\begin{array}{rl}
     & \sigma_0 \in \Sigma[ \mathbf x]_k\,, \\
     & \sigma_i \in \Sigma[ \mathbf x]_{k-\lceil g_i \rceil}\,,\\
     & \psi_j\in \R[\mathbf x]_{2(k - \lceil h_j \rceil)}
\end{array}\right.\right\}\,,\]
is its truncated version at order $k$. 
 Notice that
$g_1\,(=R-\|\mathbf x\|_2^2)\in Q(g,h)$ and therefore $Q(g,h)$ is \emph{Archimedean} \cite{lasserre2010moments}.

Let $c_\alpha := \frac{|\alpha|!}{\alpha_1!\dots\alpha_n!}$ for each $\alpha \in \N^n$. 
We note $\| p \|: = \max _\alpha   \frac{|p_\alpha|} {c_\alpha}$, for a given $p \in \R[\mathbf x]$.
As a consequence of Nie-Schweighofer's main result in \cite[Theorem 8]{nie2007complexity}, one obtains the following result:
\begin{lemma}
\label{lem:Nie-Schweighofer}
Let $f^\star$ be as in \eqref{eq:POP.on.variety.ball} with $S(g,h)\ne\emptyset$ as in \eqref{eq:variety-V.ball}.
There exists $c>0$ depending on $g$ and $h$ such that for $k\in\N$ with $k\ge c\exp((2d^2n^d)^c)$, one has
\[(f-f^\star) +6d^3n^{2d}\|f\|\log(k/c)^{-1/c}\in Q_k(g,h)\,.\]
\end{lemma}
Next, consider the hierarchy of semidefinite programs (SDP) indexed by $k\in\N$:
\begin{equation}\label{eq:sos.hierarchy.ball}
\rho_k\,:=\,\sup \{\,\xi\in\R\ :\ f-\xi \in Q_k(g,h)\}\,.
\end{equation}
By invoking Lemma \ref{lem:Nie-Schweighofer}, one obtains the convergence behavior of the sequence $(\rho_k)_{k\in\N}$ in the following result.
\begin{theorem}\label{theo:conver.semi.hie.ball}
Let $f^\star$ be as in \eqref{eq:POP.on.variety.ball} with $S(g,h)\ne\emptyset$ as in \eqref{eq:variety-V.ball}. Then:
\begin{enumerate}
\item For all $k\in\N$,  $\rho_k\le \rho_{k+1}\le f^\star$.
\item The sequence $(\rho_k)_{k\in\N}$ converges to $f^\star$ with rate 
at least 
$\mathcal{O}(\log(k/c)^{-1/c})$.
\end{enumerate}
\end{theorem}
For every $k\geq k_{\min}:=\max_{i,j}\{\lceil g_i\rceil,\,\lceil h_j\rceil\,\}$ the dual of \eqref{eq:sos.hierarchy.ball} reads
\begin{equation}\label{eq:moment.hierarchy.ball}
    \begin{array}{rl}
\tau_k \,:= \,\inf \limits_{\mathbf y \in {\R^{\stirlingii[0.7] n{2k}} }} & L_{\mathbf y}(f):\\
\qquad \text{s.t. }& \mathbf M_k(\mathbf y) \succeq 0\,;\:y_0\,=\,1\\
&\mathbf M_{k - \lceil g_i \rceil }(g_i\;\mathbf y)   \succeq 0\,,\,i\in[l_g]\,,\\
&\mathbf M_{k - \lceil h_j \rceil }(h_j\;\mathbf y)   = 0\,,\,j\in [l_h]\,.
\end{array}
\end{equation}
Strong duality between \eqref{eq:sos.hierarchy.ball} and \eqref{eq:moment.hierarchy.ball} holds
if $\tau_k=\rho_k$. 
Slater's condition on either 
\eqref{eq:sos.hierarchy.ball} or  \eqref{eq:moment.hierarchy.ball} is a well-known sufficient condition to ensure strong duality.
However, in case of equality constraints in the description \eqref{eq:variety-V.ball} of $S(g,h)$, Slater's condition does \emph{not} hold for \eqref{eq:moment.hierarchy.ball}.

 \begin{proposition} \label{prop:strong.duality.ball}
(Josz-Henrion \cite{josz2016strong})
Let $f^\star$ be as in \eqref{eq:POP.on.variety.ball} with $S(g,h)\ne\emptyset$ as in \eqref{eq:variety-V.ball}.
Strong duality of the primal-dual \eqref{eq:sos.hierarchy.ball}-\eqref{eq:moment.hierarchy.ball} holds for sufficiently large $k\in\N$, i.e., $\rho_k=\tau_k$ and $\tau_k\in\R$.
Moreover, SDP \eqref{eq:moment.hierarchy.ball} has an optimal solution.
\end{proposition}
In \cite{josz2016strong} the authors prove that the set of optimal solutions of 
\eqref{eq:moment.hierarchy.ball} is compact
and therefore \eqref{eq:moment.hierarchy.ball} has an optimal solution.
But nonexistence of an optimal solution of SDP \eqref{eq:sos.hierarchy.ball} may occur.
However, if $S(g,h)$ has nonempty interior then SDP \eqref{eq:moment.hierarchy.ball} has a strictly feasible solution and therefore Slater's condition holds.

\begin{proposition}\label{prop:slater.cond}
(Lasserre \cite[Theorem 3.4 (a)]{lasserre2001global})
If $S(g,h)$ has nonempty interior, then  Slater's condition for the primal-dual \eqref{eq:sos.hierarchy.ball}-\eqref{eq:moment.hierarchy.ball} holds for $k\geq k_{\min}$. 
In this case, $\rho_k=\tau_k$, $\tau_k\in\R$ and both primal-dual \eqref{eq:sos.hierarchy.ball}-\eqref{eq:moment.hierarchy.ball} have optimal solutions.
\end{proposition}
Let $\delta _{\mathbf a}$ stands for the Dirac measure at point $\mathbf a\in \R^n$. The following result is a consequence of of Curto-Fialkow's Flat Extension Theorem 
\cite{curto2005truncated,laurent2005revisiting}.

\begin{proposition}\label{pro:flatness.ball}
Let $\mathbf y^\star$ be an optimal solution of the SDP \eqref{eq:moment.hierarchy.ball} at some order $k\in\N$, and assume that the flat extension condition holds, i.e., $\rank(\mathbf M_{k-w}(\mathbf y^\star))=\rank(\mathbf M_{k}(\mathbf y^\star))=:r$, with $w:=\max_{i,j} \{\lceil g_i \rceil,\lceil h_j \rceil\}$.

Then $\mathbf y^\star$ has a representing $r$-atomic measure $\mu  = \sum_{t = 1}^{r} {{\lambda_j}{\delta _{  \mathbf a^{(t)} }}} $, 
where 
$(\lambda_1,\dots,\lambda_r)$ belong to standard $(r-1)$-simplex and $\{\mathbf a^{(1)},\dots,\mathbf a^{(r)}\}\subset S(g,h)$. Moreover, $\tau_k=f^\star$ and $\mathbf a^{(1)},\dots,\mathbf a^{(r)}$ are all global minimizers of POP \eqref{eq:POP.on.variety.ball}.
\end{proposition}

Henrion and Lasserre  \cite{henrion2005detecting} provide a numerical algorithm to extract the $r$ minimizer $\mathbf a^{(1)},\dots,\mathbf a^{(r)}$ from  $\mathbf M_k(\mathbf y^\star)$ when the assumptions of Proposition \ref{pro:flatness.ball} hold.

The following proposition provides a sufficient condition to ensure finite convergence of the sequence $(\tau_k)_{k\in\N}$.
\begin{proposition}\label{prop:finite.conver.ball}
The following statements are true:
\begin{enumerate}
\item (Nie \cite{nie2014optimality}) The equality $\tau_k=f^\star$ occurs generically for some $k\in\N$.
\item (
Lasserre \cite[Theorem 7.5]{lasserre2015introduction}) If (i) $Q(g,h)$ is Archimedean, (ii) the ideal $\left<h\right>$
is real radical, and (iii) the second-order sufficient condition S2 (see Definition \ref{def-S2}) holds at every global minimizer of POP  \eqref{eq:POP.on.variety.ball}, then $\tau_k=\rho_k=f^\star$ for some $k\in\N$ and both  primal-dual \eqref{eq:sos.hierarchy.ball}-\eqref{eq:moment.hierarchy.ball} have optimal solutions.
\item (Lasserre et al. \cite[Proposition 1.1]{lasserre2008semidefinite} and 
\cite[Theorem 6.13]{lasserre2015introduction}) If $V(h)$ defined as in \eqref{eq:variety-V} is finite, $\tau_k=\rho_k=f^\star$ for some $k\in\N$ and both primal-dual \eqref{eq:sos.hierarchy.ball}-\eqref{eq:moment.hierarchy.ball} have optimal solutions. In this case, the flatness condition holds at order $k$.
\end{enumerate}
\end{proposition}
 Note that the real radical property is not generic and so the condition ``$\langle h\rangle$ is real radical" must be checked case by case.
On the other hand, if $V(h)$ is the real zero set 
of a squared system of polynomial equations, i.e., $l_h=n$, then generically $V(h)$ is finite. 
\subsection{POPs on a variety contained in a sphere}
\label{sec:pop.on.sphere}
We consider a special form of POP \eqref{eq:POP.on.variety.ball} which is of the form
\begin{equation}\label{eq:POP.on.variety}
f^\star:=\inf \,\{\,f(\mathbf x)\ :\ \mathbf x\in V(h)\}\,,
\end{equation}
where $V(h)$ is 
the real variety defined by:
\begin{equation}
\label{eq:variety-V}
    V(h) \,:=\,\{\,\mathbf x\in\R^n:\: h_j(\mathbf x)\,=\,0\,;\:j=1,\ldots,l_h\,\} \,,
\end{equation}
for some set of polynomials $h:=\{h_j\}\subset\R[\mathbf x]$.
We assume that $h_1:=\bar R-\|x\|_2^2$ for some $\bar R>0$, so that
$V(h)\subset \partial B_{\bar R}^n$, where $\partial B_{\bar R}^n:=\{\mathbf x\in\R^n: {\bar R}-\Vert \mathbf x\Vert_2^2= 0\}$.
By assuming that $V(h)\ne \emptyset$, $f^\star <\infty$ and POP \eqref{eq:POP.on.variety} has at least one global minimizer.

Given $k\in\N$, define the \emph{truncated preordering} of order $k$ associated with the variety $V(h)$ in \eqref{eq:variety-V} as follows:
\[P_k(h): = \left\{ \sigma_0+ \sum_{j = 1}^{l_h} {\psi _jh_j} \ :\ \sigma_0 \in \Sigma[ \mathbf x]_k\,,\, \psi_j\in \R[\mathbf x]_{2(k - \lceil h_j \rceil)}\,,\,j\in[l_h]\right\}\,.\]
\begin{remark}
For every $k\in\N$, $P_k(h)$ is also the truncated quadratic module $Q_k(h)$ associated with the semialgebraic set $V(h)=S(\emptyset,h)$.
\end{remark}

As a consequence of Schweighofer's main result in \cite[Theorem 4]{schweighofer2004complexity}, one obtains the following result:
\begin{lemma}
\label{lem:Schweighofer}
Let $f^\star$ be as in \eqref{eq:POP.on.variety} with $V(h)$ as in \eqref{eq:variety-V}.
There exists $c>0$ depending on $V$ such that for $k\in\N$ with $k\ge cd^cn^{cd}$, one has
\[(f-f^\star) +cd^4n^{2d}\|f\|k^{-1/c}\in  P_k(h)\,.\]
\end{lemma}
Note that in the case of polynomial optimization on the sphere, one can take $c = 1$ in Lemma \ref{lem:Schweighofer}, as a consequence of the convergence result from  \cite{doherty2012convergence}.

Next, consider the hierarchy of semidefinite programs (SDP) indexed by $k\in\N$:
\begin{equation}\label{eq:sos.hierarchy}
\rho_k\,:=\,\sup \,\{\,\xi\in\R\ :\ f-\xi \in P_k(h)\}\,.
\end{equation}
For every $k\in\N$, the dual of \eqref{eq:sos.hierarchy} reads
\begin{equation}\label{eq:moment.hierarchy}
    \begin{array}{rl}
\tau_k := \inf \limits_{\mathbf y \in {\R^{\stirlingii [0.7]n {2k}} }} & L_{\mathbf y}(f)\\
\qquad \text{s.t. }& \mathbf M_k(\mathbf y) \succeq 0\,;\: y_0\,=\,1\\
&\mathbf M_{k - \lceil h_j \rceil }(h_j\;\mathbf y)   = 0\,,\,j\in [l_h]\,.
\end{array}
\end{equation}
By invoking Lemma \ref{lem:Schweighofer}, one obtains the convergence behavior of the sequence $(\rho_k)_{k\in\N}$ in the following result.
\begin{theorem}\label{theo:conver.semi.hie}
Let $f^\star$ be as in \eqref{eq:POP.on.variety} with $V(h)\neq\emptyset$ as in \eqref{eq:variety-V}. Then:
\begin{enumerate}
\item For all $k\in\N$, $\rho_k\le \rho_{k+1}\le f^\star$.
\item The sequence $(\rho_k)_{k\in\N}$ converges to $f^\star$ with rate at least 
$\mathcal{O}(k^{-1/c})$.
\item If the ideal $\langle h\rangle$ is real radical and the second-order sufficiency condition S2 (Definition \ref{def-S2}) holds at every global minimizer of POP \eqref{eq:POP.on.variety} then $\tau_k=\rho_k=f^\star$ for some $k$ and \eqref{eq:sos.hierarchy} has an optimal solution, i.e., $f-f^\star\in P_k(h)$.
\item If $V(h)$ defined as in \eqref{eq:variety-V} is finite, $\tau_k=\rho_k=f^\star$ for some $k\in\N$ and both primal-dual \eqref{eq:sos.hierarchy.ball}-\eqref{eq:moment.hierarchy.ball} have optimal solutions. In this case, the flatness condition holds at order $k$.
\end{enumerate}
\end{theorem}
With $V(h)$ in lieu of $S(g,h)$,
strong duality and analogues of
Proposition \ref{prop:strong.duality.ball} and
\ref{pro:flatness.ball}, 
also hold.

\subsection{Spectral minimizations of SDP}
\label{sec:sdp.ctp.btp}
Let $s,l, s^{j}\in\N^{\ge 1}$, $j\in[l]$, be fixed such that $s=\sum_{j=1}^l s^{(j)}$. 
Let $\mathcal{S}$ be the set of real symmetric matrices of size $s$ in a block diagonal form:
\begin{equation}\label{eq:block.diagonal.form}
    \mathbf X=\diag(\mathbf X_1,\dots,\mathbf X_l)\,,
\end{equation}
such that $\mathbf X_j$ is of size $s^{(j)}$, $j\in [l]$.
Let $\mathcal{S^+}$ be the set of all $\mathbf X\in \mathcal S$ such that $\mathbf X\succeq 0$, i.e., $\mathbf X$ has only nonnegative eigenvalues.
Then $\mathcal S$ is a Hilbert space with scalar product $\left<\mathbf A, \mathbf B\right> = \trace(\mathbf B^T \mathbf A)$ and $\mathcal S^+$ is a self-dual cone.

Let us consider the following SDP:
\begin{equation}\label{eq:SDP.form.0}
-\tau = \sup_{\mathbf X\in \mathcal{S}} \,\{ \,\left< \mathbf C,\mathbf X\right>\,:\,\mathcal{A} \mathbf X=\mathbf b\,,\, \mathbf X \succeq 0\,\}\,,
\end{equation}
where $\mathcal{A}:\mathcal{S}\to \R^{m}$ is a linear operator of the form
\[\mathcal{A}\mathbf X=\left[\left< \mathbf A_{1},\mathbf X\right>,\dots,\left< \mathbf A_{m},\mathbf X\right>\right]\,,\]
with $\mathbf A_{i} \in \mathcal{S}$, $i \in[m]$, $\mathbf C \in \mathcal{S}$ is the cost matrix and $\mathbf b\in \R^{m}$ is the right-hand-side vector. 

The dual of SDP \eqref{eq:SDP.form.0} reads:
\begin{equation}\label{eq:SDP.form.dual.0}
-\rho = \inf_{\mathbf z}\, \{\, \mathbf b^T\mathbf z\,:\,
\mathcal{A}^T \mathbf z-\mathbf C\succeq 0\,\}\,,
\end{equation}
where $\mathcal{A}^T:\R^{m}\to \mathcal{S}$ is the adjoint operator of $\mathcal{A}$, i.e., $\mathcal{A}^T\mathbf z=\sum_{i=1}^{m} z_i\mathbf A_{i}$. 
The following assumption will be used in the next two sections:
\begin{assumption}\label{ass:general.assump.sdp}
Consider the following conditions:
\begin{enumerate}
\item Strong duality of primal-dual \eqref{eq:SDP.form.0}-\eqref{eq:SDP.form.dual.0} holds, i.e., $\tau=\rho$ and $\tau\in\R$.
\item Primal attainability: SDP \eqref{eq:SDP.form.0} has an optimal solution.
\item Dual attainability: SDP \eqref{eq:SDP.form.dual.0} has an optimal solution.
\item Constant trace property (CTP): There exists $a>0$ such that
\begin{equation}\label{eq:constan.trace.prop}
\forall \ \mathbf X\in \mathcal S\,,\,\mathcal{A} \mathbf X=\mathbf b\Rightarrow 
\trace(\mathbf X)=a\,.
\end{equation}
\item Bounded trace property (BTP): There exists $a>0$ such that
\begin{equation}\label{eq:constan.trace.prop.btp}
\forall \ \mathbf X\in \mathcal S\,,\,\mathcal{A} \mathbf X=\mathbf b\Rightarrow 
 \trace(\mathbf X) \le  a\,.
\end{equation}
\end{enumerate}
\end{assumption}
In Assumption \ref{ass:general.assump.sdp}, conditions 1 and 5 (or 4) imply condition 2. Indeed, if condition 5 holds, the feasible set of \eqref{eq:SDP.form.0} is compact and if condition 1 holds, the feasible set of \eqref{eq:SDP.form.0} is nonempty.
Moreover, condition 2 and 5 (or 4) imply condition 1. Indeed, if condition 2 and 5 hold,
the set of optimal solutions of \eqref{eq:SDP.form.0} is nonempty and bounded. 
Then Trnovska's result \cite[Corollary 1]{trnovska2005strong} yields the desired conclusion.
\begin{remark}\label{re:convert.BTP.CTP}
If condition 5 of Assumption \ref{ass:general.assump.sdp} holds, by adding a slack variable $y$ and noting $\mathbf Y= \diag(\mathbf X,y)$, we obtain an equivalent SDP of \eqref{eq:SDP.form.0} as follows: 
\begin{equation}\label{eq:add.slack.variable}
-\tau = \sup_{\mathbf Y\in \hat{\mathcal{ S}}} \{ \left< \hat {\mathbf C},\mathbf Y\right>\,:\,\hat{\mathbf A}_i \mathbf Y=\mathbf b_i\,,\, \mathbf Y \succeq 0\,,\, \trace(\mathbf Y)=a\}\,,
\end{equation}
where $\hat{\mathcal{ S}}=\{\diag(\mathbf X,y): X\in \mathcal{ S}\,,\, y\in\R\}$, $\hat {\mathbf C}=\diag(\mathbf C,0)$ and $\hat {\mathbf A}_i=\diag(\mathbf A_i,0)$.
Obviously, SDP \eqref{eq:add.slack.variable} has CTP.
\end{remark}

\subsubsection{SDP with Constant Trace Property (CTP)}
\label{sec:sdp.ctp}
 Recall that 
$\lambda_1(\mathbf A)$ stands for the largest eigenvalue of a real symmetric matrix $\mathbf A$.

\begin{lemma} \label{lem:obtain.dual.sol}
Let conditions 1 and 4 
of Assumption \ref{ass:general.assump.sdp} hold and let
$\varphi:\R^{m}\to \R$ be the function:
\begin{equation}\label{eq:func.phik.0}
\mathbf z\mapsto \varphi(\mathbf z)\,:=\,a\lambda_1(\mathbf C-\mathcal{A}^T\mathbf z)+\mathbf b^T\mathbf z\,.
\end{equation}
Then:
\begin{equation}\label{eq:nonsmooth.hierarchy.0}
-\tau=\inf_{\mathbf{z}}\,\{\,\varphi(\mathbf z)\ :\ \mathbf z\in \R^{m}\}\,.
\end{equation}
Moreover if condition 3 
of Assumption \ref{ass:general.assump.sdp} holds, i.e., SDP \eqref{eq:SDP.form.dual.0} has an optimal solution then problem \eqref{eq:nonsmooth.hierarchy.0} has an optimal solution.
\end{lemma}
The proof of Lemma \ref{lem:obtain.dual.sol} is postponed to Appendix \ref{proof:lem:obtain.dual.sol}.

Given $r\in\N^{\ge 1}$ and $\mathbf u_j\in \R^s$, $j\in[r]$, consider the following convex quadratic optimization problem (QP):
    \begin{equation}\label{eq:socp.sdp.sol}
    \begin{array}{rl}
        \min\limits_{ \xi\in\R^r} &\frac{1}2\left\|\mathbf b- a\mathcal{A}\left(\sum_{j=1}^r\xi_j\mathbf u_j\mathbf u_j^T\right)\right\|_2^2\\
        \text{s.t.}& \sum_{j=1}^r \xi_j=1\,;\:\xi_j\ge 0\,,\,j\in [r]\,.
        \end{array}
    \end{equation}

Next, we describe 
Algorithm \ref{alg:sol.SDP.CTP.0} to solve  SDP \eqref{eq:SDP.form.0}, which is based on nonsmooth first-order optimization methods (e.g., LMBM \cite[Algorithm 1]{haarala2007globally}).
As shown later on in Section \ref{sec:benchmark}, 
this algorithm works well in almost all cases and with significantly lower computational cost
when compared to the (currently fastest) SDP solver Mosek 9.1.

    \begin{algorithm}
    \footnotesize
    \caption{SDP-CTP}
    \label{alg:sol.SDP.CTP.0} 
    \textbf{Input:} SDP \eqref{eq:SDP.form.0} with unknown optimal value and optimal solution;\\
    \hspace*{\algorithmicindent}\hspace*{\algorithmicindent} method (T) for solving convex nonsmooth unconstrained optimization problems (NSOP). \\
    \textbf{Output:} optimal value $-\tau$ and optimal solution $\mathbf X^\star$ of SDP \eqref{eq:SDP.form.0}.
    \begin{algorithmic}[1]
    \State Compute the optimal value $-\tau$ and an optimal solution $\mathbf{\bar z}$ of the  NSOP \eqref{eq:nonsmooth.hierarchy.0} by using method (T);
 \State Compute $\lambda_1(\mathbf C-\mathcal{A}^T\mathbf{\bar z})$ and its corresponding uniform eigenvectors $\mathbf u_1,\dots,\mathbf u_r$;
 \State Compute an optimal solution $(\bar \xi_1,\dots,\bar \xi_r)$ of QP \eqref{eq:socp.sdp.sol} and set $\mathbf X^\star=a\sum_{j=1}^r \bar \xi_j\mathbf u_j\mathbf u_j^T$.
    \end{algorithmic}
    \end{algorithm}

The fact that Algorithm \ref{alg:sol.SDP.CTP.0} is well-defined
under certain conditions is a corollary of Lemma \ref{lem:obtain.dual.sol},  \ref{lem:extract.sdp.solu} and \ref{lem:LMBM.conver.guara}.
\begin{corollary}\label{coro:well-defined.sdp.by.nonsmooth}
Let conditions 1 and 4 of Assumption \ref{ass:general.assump.sdp} hold. Assume that the method (T) is globally convergent for NSOP \eqref{eq:nonsmooth.hierarchy.0} (e.g., (T) is LMBM).
Then output $-\tau$ of Algorithm \ref{alg:sol.SDP.CTP.0} is well-defined.
Moreover, if condition 3 of Assumption \ref{ass:general.assump.sdp} holds, the vector $\bar {\mathbf z}$ mentioned at Step 1 of Algorithm \ref{alg:sol.SDP.CTP.0} exists and thus the output $\mathbf X^\star$ of Algorithm \ref{alg:sol.SDP.CTP.0} is well-defined.
\end{corollary}

\paragraph{Largest eigenvalue computation:}
\label{para:eigen.comp}
Step 1 of Algorithm \ref{alg:sol.SDP.CTP.0} (resp. Algorithm \ref{alg:sol.SDP.CTP.0.btp})
requires the largest eigenvalue and corresponding eigenvectors of $\mathbf C-\mathcal{A}^T\mathbf z$ to evaluate the function $\varphi$ (resp. $\psi$) and a subgradient of the subdifferential $\partial \varphi$ (resp. $\partial \psi$) given in Proposition \ref{prop:properties.phi.0} (resp. Proposition \ref{prop:properties.phi.0.btp}) at $\mathbf z$. 
Fortunately, solving the eigenvalue problem for $\mathbf C-\mathcal{A}^T\mathbf z\in \mathcal S$ can be done on every block of $\mathbf C-\mathcal{A}^T\mathbf z$. 
Indeed, with $\mathbf X\in \mathcal S$ as in \eqref{eq:block.diagonal.form}, 
\[\lambda(\mathbf X)=\lambda(\mathbf X_1)\cup \dots\cup\lambda(\mathbf X_l)\,,\]
where $\lambda(\mathbf A)$ is the set of all eigenvalues $\lambda_1(\mathbf A)\ge \dots \ge \lambda_{t}(\mathbf A)$ for every real symmetric matrix $\mathbf A$ of size $t$. 
In particular, 
\[\lambda_1(\mathbf X)=\max\{\lambda_1(\mathbf X_1), \dots,\lambda_1(\mathbf X_l)\}\,.\]
If $\mathbf u\in \R^{s^{(j)}}$ is an eigenvector of $\mathbf X_j$ corresponding to the eigenvalue $\lambda_i(\mathbf X_j)$ for some $i\in[s^{(j)}]$ and $j\in[l]$, by adding zeros entries in $\mathbf u$, 
\[\bar {\mathbf u}=(\mathbf 0_{\R^{s^{(1)}+\dots+s^{(j-1)}}},\mathbf u,\mathbf 0_{\R^{s^{(j+1)}+\dots+s^{(l)}}})\]
is an eigenvector of $\mathbf X=\diag(\mathbf X_1,\dots,\mathbf X_l)$ corresponding to  $\lambda_i(\mathbf X_j)$. 

The interested reader can refer to Lanczos algorithm in \cite{lanczos1950iteration} and its modified version \cite{ojalvo1970vibration} to solve largest eigenvalue problems of symmetric matrices of large sizes.

\begin{remark}
Let conditions 1, 2 and 5 of Assumption \ref{ass:general.assump.sdp} hold. 
We keep all notation from Remark \ref{re:convert.BTP.CTP}. 
By applying Lemma \ref{lem:obtain.dual.sol} for SDP \eqref{eq:add.slack.variable} with CTP, one has
\begin{equation}\label{eq:app.Lem.CTP}
-\tau=\inf\{\,\,a\lambda_1(\hat {\mathbf C}-\hat {\mathcal{A}}^T\mathbf z)+\mathbf b^T\mathbf z\ :\ \mathbf z\in \R^{m}\}\,,
\end{equation}
where $\hat {\mathcal{A}}^T\mathbf z=\sum_{i=1}^{m} z_i\hat {\mathbf A}_{i}$. 
Note that $\hat {\mathbf C}-\hat {\mathcal{A}}^T\mathbf z=\diag(\mathbf C- {\mathcal{A}}^T\mathbf z,0)$.
It implies that $\lambda_1(\hat {\mathbf C}-\hat {\mathcal{A}}^T\mathbf z)=\max\{\lambda_1(\mathbf C- {\mathcal{A}}^T\mathbf z),0\}$.
Thus, \eqref{eq:app.Lem.CTP} can be rewritten as 
\begin{equation}\label{eq:pre.BTP}
-\tau=\inf\{\,\,a\max\{\lambda_1(\mathbf C- {\mathcal{A}}^T\mathbf z),0\}+\mathbf b^T\mathbf z\ :\ \mathbf z\in \R^{m}\}\,.
\end{equation}

\end{remark}
In the next section, we consider the spectral formulation \eqref{eq:pre.BTP} introduced by Ding et al. in \cite[Section 6]{ding2019optimal}.
\subsubsection{SDP with Bounded Trace Property (BTP)}
\label{sec:sdp.btp}
In the last subsection, we have seen that SDPs with CTP can be solved efficiently with first-order methods. 
Similar results 
can be obtained for the larger class of SDPs with
the weaker \emph{bounded trace property} (BTP).
In particular the semidefinite relaxations 
of the Moment-SOS hierarchy associated with
a POP on a compact semialgebraic set have 
the BTP. 
So in principle there is no need
to add auxiliary ``slack" variables  to obtain an equivalent CTP-POP, as shown in Remark \ref{re:convert.BTP.CTP}. 
However, numerical experiments of Section \ref{sec:benchmark} suggest that the CTP is 
a highly desirable property 
that justifies addition of auxiliary variables.

 The analogue of Lemma \ref{lem:obtain.dual.sol} for BTP reads:
\begin{lemma} \label{lem:obtain.dual.sol.btp}
Let conditions 1, 2 and 5 of Assumption \ref{ass:general.assump.sdp} hold, and let $\psi:\R^{m}\to \R$ be the function:
\begin{equation}\label{eq:func.phik.0.btp}
\mathbf z\mapsto \psi(\mathbf z)\,:=\, a\max\{\lambda_1(\mathbf C-\mathcal{A}^T\mathbf z),0\}+\mathbf b^T\mathbf z\,.
\end{equation}
Then
\begin{equation}\label{eq:nonsmooth.hierarchy.0.btp}
-\tau=\inf_{\mathbf{z}}\,\{\,\psi(\mathbf z)\ :\ \mathbf z\in \R^{m}\}\,.
\end{equation}
Moreover if condition 3 of Assumption \ref{ass:general.assump.sdp} holds, then problem \eqref{eq:nonsmooth.hierarchy.0.btp} has an optimal solution.
\end{lemma}
The proof of Lemma \ref{lem:obtain.dual.sol.btp} is postponed to Appendix \ref{proof:lem:obtain.dual.sol.btp}.

Given $r\in\N^{\ge 1}$, $\mathbf u_j\in \R^s$, $j\in[r]$ and $\bar{\mathbf z}\in\R^m$, consider the convex quadratic optimization problem (QP):
    \begin{equation}\label{eq:socp.sdp.sol.btp}
    \begin{array}{rl}
        \min\limits_{ \xi\in\R^r} &\frac{1}2\left\|\mathbf b- \mathcal{A}\left(\sum_{j=1}^r\xi_j\mathbf u_j\mathbf u_j^T\right)\right\|_2^2\\
        \text{s.t.}& \xi_j\ge 0\,,\,j\in [r]\,, \\
        &\sum_{j=1}^r \xi_j \begin{cases} 
        =0 &\text{ if } \lambda_1(\mathbf C-\mathcal{A}^T\mathbf{\bar z})< 0\,,\\
        \le a &\text{ if } \lambda_1(\mathbf C-\mathcal{A}^T\mathbf{\bar z})= 0\,.\\
        = a &\text{ otherwise}\,.
        \end{cases}
        \end{array}
    \end{equation}

We next describe 
Algorithm \ref{alg:sol.SDP.CTP.0.btp} to solve  SDP \eqref{eq:SDP.form.0}. As Algorithm
\ref{alg:sol.SDP.CTP.0}, it is also based on 
nonsmooth optimization methods such as LMBM. 

    \begin{algorithm}
    \footnotesize
    \caption{SDP-BTP}
    \label{alg:sol.SDP.CTP.0.btp} 
    \textbf{Input:} SDP \eqref{eq:SDP.form.0} with unknown optimal value and optimal solution;\\
    \hspace*{\algorithmicindent}\hspace*{\algorithmicindent} method (T) for solving convex NSOP. \\
    \textbf{Output:} optimal value $-\tau$ and optimal solution $\mathbf X^\star$ of SDP \eqref{eq:SDP.form.0}.
    \begin{algorithmic}[1]
    \State Compute the optimal value $-\tau$ and an optimal solution $\mathbf{\bar z}$ of NSOP \eqref{eq:nonsmooth.hierarchy.0.btp} by using method (T);
 \State Compute $\lambda_1(\mathbf C-\mathcal{A}^T\mathbf{\bar z})$ and its corresponding uniform eigenvectors $\mathbf u_1,\dots,\mathbf u_r$;
 \State Compute an optimal solution $(\bar \xi_1,\dots,\bar \xi_r)$ of QP \eqref{eq:socp.sdp.sol.btp}  and set $\mathbf X^\star=\sum_{j=1}^r \bar \xi_j\mathbf u_j\mathbf u_j^T$.
    \end{algorithmic}
    \end{algorithm}

The next result is a consequence of 
Lemma \ref{lem:obtain.dual.sol.btp},  \ref{lem:extract.sdp.solu.btp} and \ref{lem:LMBM.conver.guara.btp}.

\begin{corollary}\label{coro:well-defined.sdp.by.nonsmooth.btp}
Let conditions 1, 2 and 5 of Assumption \ref{ass:general.assump.sdp} hold. Assume that method (T) is globally convergent for NSOP \eqref{eq:nonsmooth.hierarchy.0.btp} (e.g., (T) is LMBM).
Then the output $-\tau$ of Algorithm \ref{alg:sol.SDP.CTP.0.btp} is well-defined.
Moreover, if condition 3 of Assumption \ref{ass:general.assump.sdp} holds, the output $\mathbf X^\star$ of Algorithm \ref{alg:sol.SDP.CTP.0.btp}  is well-defined.
\end{corollary}
\section{Applications}
\label{sec:applica.SDP-CTP}
\subsection{Polynomial optimization}
\label{sec:POP.on.compact.set}
We consider the following POP:
\begin{equation}
f^\star:=\inf \{f(\mathbf x)\ :\ \mathbf x\in S(g,h)\}\,,
\end{equation}
where $S(g,h)$ is defined as in \eqref{eq:variety-V.ball} with $l_g$ (resp. $l_h$) being the number of inequality (resp. equality) constraints.
Assume that $S(g,h)\subset B_R^n$. 
\begin{remark}
\label{re:BTP.for.POP}
By setting $\mathbf X:=\diag(\mathbf M_k(\mathbf y),\mathbf M_{k-\lceil g_1 \rceil}(g_1\mathbf y),\dots,\mathbf M_{k-\lceil g_{l_g} \rceil}(g_{l_g}\mathbf y))$ and using the upper bound $\trace(\mathbf X)\le \bar a_k$ with
\begin{equation}\label{eq:bound.trace.BTP.POP}
    \bar a_k:= R^k\left( \stirlingii {n+k}{n}+\sum_{i=1}^{l_g} \|g_i\|_1\stirlingii {n+k-\lceil g_i\rceil}{n}\right)\,,
\end{equation}
SDP \eqref{eq:moment.hierarchy.ball} can be converted to an equivalent SDP with BTP, thanks to the absolute upper bound for each moment variable $|y_\alpha|\le R^{|\alpha|/2}$, $\alpha\in\N^n$.
In principle, we can solve this SDP by applying directly Algorithm \ref{alg:sol.SDP.CTP.0.btp}. 
However, in our experiments presented in Section \ref{sec:benchmark} this method is not only inefficient but also provides output with low accuracy.
\end{remark}
In order to overcome the accuracy issue mentioned in Remark \ref{re:BTP.for.POP}, we convert every POP 
to a CTP-POP (i.e., a new POP formulation with CTP) by adding
slack variables associated with inequality constraints.
In the sequel, we consider three particular cases: equality constrained POPs on a sphere in Section \ref{sec:POP.sphere}, constrained POPs with single inequality (ball) constraint in Section \ref{sec:POP.compact.case1}, and constrained POPs on a ball in Section \ref{sec:POP.compact.case2}.
\subsubsection{Equality constrained POPs on a sphere}
\label{sec:POP.sphere}
Assume that $l_g=0$ and $h_1=\bar R-\|\mathbf x\|_2^2$.
In this case, we consider equality constrained POPs on a sphere, presented in Section \ref{sec:pop.on.sphere}.
We propose to reduce SDP \eqref{eq:moment.hierarchy} to an NSOP.
For each $k\in\N^n_k$, let $(\theta_{k,\alpha})_{\alpha\in\N^n_k}$ be the finite sequence of positive real numbers such that 
\[(1+\|\mathbf
x\|_2^2)^k=\sum_{\alpha\in\N^n_k}\theta_{k,\alpha}x^{2\alpha}\,,\]
and define the diagonal matrix 
\begin{equation}\label{eq:P.mat}
\mathbf P_{n,k}:=\diag((\theta_{k,\alpha}^{1/2})_{\alpha\in\N^n_k})\,.
\end{equation}
For every $k\in\N$, since $\mathbf P_{n,k}\succ 0$, SDP \eqref{eq:moment.hierarchy} is equivalent to SDP:
\begin{equation}\label{eq:moment.hierarchy.multiply.diagonal}
    \begin{array}{rl}
\tau_k = \inf \limits_{\mathbf y \in {\R^{\stirlingii [0.7]n {2k}} }} & L_{\mathbf y}(f)\\
\qquad \text{s.t. }& y_0\,=\,1\,;\:\mathbf P_{n,k}\mathbf M_k(\mathbf y)\mathbf P_{n,k} \succeq 0\,,\\
&\mathbf M_{k - \lceil h_j \rceil }(h_j\;\mathbf y)   = 0\,,\,j\in[l_h]\,.
\end{array}
\end{equation}

For every $k\in\N$, note $a_k:=({\bar R}+1)^{k}$. We will use the following lemma:
\begin{lemma} \label{lem:trace.constant.property}
For all $k\in\N$,
\[\left.
\begin{array}{rl}
&\mathbf M_{k - 1 }(({\bar R}-\|\mathbf x\|_2^2)\;\mathbf y)   = 0\,,\\
&y_0=1
\end{array}
\right\}\Rightarrow \trace(\mathbf P_{n,k} \mathbf M_k(\mathbf y)\mathbf P_{n,k})=a_k\,.\]\end{lemma}
\begin{proof}
Let $k\in\N$ be fixed. 
From $\mathbf M_{k - 1 }(({\bar R}-\|\mathbf x\|_2^2)\;\mathbf y)   = 0$, $L_{\mathbf y}(p({\bar R}-\|\mathbf x\|_2^2))=0$, for every $p\in\R[\mathbf x]_{2(k-1)}$. For every $r\in\N^{\le k-1}$, by choosing $p=\|\mathbf x\|_2^{2r}$, 
\[L_{\mathbf y}(\|\mathbf x\|_2^{2(r+1)})=-L_{\mathbf y}(\|\mathbf x\|_2^{2r}({\bar R}-\|\mathbf x\|_2^2))+{\bar R}L_{\mathbf y}(\|\mathbf x\|_2^{2r})={\bar R}L_{\mathbf y}(\|\mathbf x\|_2^{2r})\,.\]
By induction, $L_{\mathbf y}(\|\mathbf x\|_2^{2r})={\bar R}L_{\mathbf y}(\|\mathbf x\|_2^{2(r-1)})=\dots={\bar R}^kL_{\mathbf y}(\|\mathbf x\|_2^{2\times 0})={\bar R}^ky_0={\bar R}^r$, for every $r\in\N^{\le k}$. Thus,
\[\begin{array}{rl}
\trace(\mathbf P_{n,k} \mathbf M_k(\mathbf y)\mathbf P_{n,k})&=\displaystyle\sum_{\alpha\in \N^n_k} \theta_{k,\alpha}^{1/2}y_{2\alpha}\theta_{k,\alpha}^{1/2}=L_{\mathbf y}\left(\displaystyle\sum_{\alpha\in \N^n_k} \theta_{k,\alpha} x^{2\alpha}\right)\\
&=L_{\mathbf y}((1+\|\mathbf x\|_2^2)^k)=L_{\mathbf y}\left(\displaystyle\sum_{r=0}^k\binom{k}{r}\|\mathbf x\|_2^{2r}\right)\\
&=\displaystyle\sum_{r=0}^k\binom{k}{r}L_{\mathbf y}(\|\mathbf x\|_2^{2r})=\sum_{r=0}^k\binom{k}{r}{\bar R}^r=({\bar R}+1)^{k}\,.
\end{array}\]
\end{proof}
For each $k\in\N$, let us denote by $\mathcal{S}_k$ the set of symmetric matrices of size $\omega_k =\stirlingii n k$ and let $\left<\mathbf A, \mathbf B\right> = \trace(\mathbf B^T \mathbf A)$
be the usual scalar product on $\mathcal{S}_k$.
For every $k\in\N$, letting 
\begin{equation}\label{eq:convert.momentmat}
\mathbf X=\mathbf P_{n,k}\mathbf M_k(\mathbf y)\mathbf P_{n,k}\,,
\end{equation}
\eqref{eq:moment.hierarchy.multiply.diagonal} can be written 
in the form:
\begin{equation}\label{eq:SDP.form}
-\tau_k = \sup _{\mathbf X\in \mathcal{S}_k} \{ \left< \mathbf C_k,\mathbf X\right>\,:\,\mathcal{A}_k \mathbf X=\mathbf b_k\,,\, \mathbf X \succeq 0\}\,,
\end{equation}
where $\mathcal{A}_k:\mathcal{S}_k\to \R^{m_k}$ is a linear operator of the form
\[\mathcal{A}_k\mathbf X=\left[\left< \mathbf A_{k,1},\mathbf X\right>,\dots,\left< \mathbf A_{k,m_k},\mathbf X\right>\right]\,,\]
with $\mathbf A_{k,i} \in \mathcal{S}_k$, $i\in[m_k]$, $\mathbf C_k \in \mathcal{S}_k$ is the cost matrix and $\mathbf b_k\in \R^{m_k}$ is the right-hand-side vector. 
Appendix \ref{sec:convert.SDP.sphere} describes how to reduce SDP \eqref{eq:moment.hierarchy.multiply.diagonal} to the form \eqref{eq:SDP.form}.

For every $k\in \N$, the dual of SDP \eqref{eq:SDP.form} reads:
\begin{equation}\label{eq:SDP.form.dual}
-\rho_k = \inf _{\mathbf z} \,\{ \,\mathbf b_k^T\mathbf z\,:\,
\mathcal{A}_k^T \mathbf z-\mathbf C_k\succeq 0\,,\}
\end{equation}
where $\mathcal{A}_k^T:\R^{m_k}\to \mathcal{S}_k$ is the adjoint operator of $\mathcal{A}_k$, i.e., $\mathcal{A}_k^T\mathbf z=\sum_{i=1}^{m_k} z_i\mathbf A_{k,i}$.

From Lemma \ref{lem:trace.constant.property} and since $h_1={\bar R}-\|\mathbf x\|_2^2$, it implies that for every $k\in\N$,
\begin{equation}\label{eq:unit.trace.prop}
    \forall \ \mathbf X\in \mathcal S_k\,,\,\mathcal{A}_k \mathbf X=\mathbf b_k\Rightarrow \trace(\mathbf X)=a_k\,.
\end{equation}

We guarantee the strong duality, primal attainability, and dual attainability for primal-dual \eqref{eq:SDP.form}-\eqref{eq:SDP.form.dual} in the following proposition:
\begin{proposition}\label{prop:strong.duality.attainability.POP.sphere}
Let $f^\star$ be as in \eqref{eq:POP.on.variety}. Then:
\begin{enumerate}
    \item Strong duality holds for primal-dual \eqref{eq:SDP.form}-\eqref{eq:SDP.form.dual}  for large enough $k\in\N$.
    \item SDP \eqref{eq:SDP.form} has an optimal solution for large enough $k\in\N$.
    \item Assume that one of the following two conditions holds:
    \begin{enumerate}
    \item $\left<h\right>$ is real radical and the second-order sufficiency condition S2  holds at every global minimizer of  \eqref{eq:POP.on.variety};
\item $V(h)$ is finite.
\end{enumerate}
 Then  SDP \eqref{eq:SDP.form.dual} has an optimal solution for large enough $k\in\N$. In this case, $\ushort \tau_k=\ushort \rho_k=f^\star$.
 \end{enumerate}
 \end{proposition}
\begin{proof}
Since \eqref{eq:moment.hierarchy} (resp. \eqref{eq:sos.hierarchy}) and \eqref{eq:SDP.form} (resp. \eqref{eq:SDP.form.dual}) are equivalent, the first and second statements follow from  Proposition \ref{prop:strong.duality.ball}. 
The third statement is due to Theorem \ref{theo:conver.semi.hie}.
\end{proof}

By replacing $(\mathcal{A}_k, \mathbf{A}_{k,i},  \mathbf b_k, \mathbf C_k, \mathcal{S}_k, \omega_k, m_k, \tau_k, \rho_k, a_k)$ by $(\mathcal{A}, \mathbf{A}_{i}, \mathbf b, \mathbf C, \mathcal{S}, s, m, \tau, \rho, a)$, primal-dual \eqref{eq:SDP.form}-\eqref{eq:SDP.form.dual} becomes primal-dual \eqref{eq:SDP.form.0}-\eqref{eq:SDP.form.dual.0}, we then go back to Section \ref{sec:sdp.ctp.btp} with $l=1$.

We illustrate the conversion from SDP \eqref{eq:moment.hierarchy} to SDP \eqref{eq:SDP.form} in the following example.
\begin{example}
Consider a simple example of POP \eqref{eq:POP.on.variety} with $n=1$:  
\[-1=\inf\{x\ :\ 1-x^2=0\}\,.\]
 Then the second order moment relaxation ($k=2$) has the form:
\[\begin{array}{rl}
\tau_2 = \inf \limits_{\mathbf y} & y_1\\
\qquad \text{s.t. }
&\begin{bmatrix}
y_0 & y_1 & y_2\\
y_1 & y_2 & y_3\\
y_2 & y_3 & y_4
\end{bmatrix}\succeq 0\,,\,
\begin{bmatrix}
y_0-y_2 & y_1-y_3 \\
y_1-y_3 & y_2-y_4 
\end{bmatrix}= 0\,,\,
 y_0=1\,.
\end{array}\]
It can be rewritten as
\[\begin{array}{rl}
\tau_2 = \inf \limits_{\mathbf y} & y_1\\
\qquad \text{s.t. }
&\begin{bmatrix}
1 & y_1 & 1\\
y_1 & 1 & y_1\\
1 & y_1 & 1
\end{bmatrix}\succeq 0\,,\\
\end{array}\]
by removing equality constraints. Obviously, the positive semidefinite matrix of this form has trace 3. 

In a different way, according to Appendix \ref{sec:convert.SDP.sphere}, let us note
\[\mathbf X=\begin{bmatrix}
1 & 0 & 0\\
0 & \sqrt{2} & 0\\
0 & 0 & 1
\end{bmatrix}
\begin{bmatrix}
y_0 & y_1 & y_2\\
y_1 & y_2 & y_3\\
y_2 & y_3 & y_4
\end{bmatrix}\begin{bmatrix}
1 & 0 & 0\\
0 & \sqrt{2} & 0\\
0 & 0 & 1
\end{bmatrix}
\,,\]
to obtain
\[-\tau_2 = \sup _{\mathbf X\in \mathcal{S}_2} \{ \left< \mathbf C,\mathbf X\right>\,:\,\left< \mathbf A_i,\mathbf X\right>=b_i\,,\,i\in [5]\,,\, \mathbf X \succeq 0\}\,,\]
where $b_1=\dots=b_4=0$, $b_5=1$ and
\[\begin{array}{rl}
&\mathbf C=-\frac{\sqrt{2}}4\begin{bmatrix}
0 &1& 0\\
1 &0& 0\\ 
0& 0& 0
\end{bmatrix}\,,\,
\mathbf A_1=\frac{\sqrt{2}}2\begin{bmatrix}
0& 0& 1\\
0& -1& 0\\
1& 0& 0
\end{bmatrix}\,,\,
\mathbf A_2=\frac{1}2\begin{bmatrix}
2& 0& -1\\
0& 0& 0\\ 
-1& 0& 0
\end{bmatrix}\,,\\
&\mathbf A_3=\frac{\sqrt{2}}4\begin{bmatrix}
0& 1& 0\\
1& 0& -1\\
0& -1& 0
\end{bmatrix}\,,\,
\mathbf A_4=\frac{1}2\begin{bmatrix}
0& 0& 1\\
0& 0& 0\\
1& 0& -2
\end{bmatrix}\,,\,
\mathbf A_5=\begin{bmatrix}
1& 0& 0\\
0& 0& 0\\
0& 0& 0
\end{bmatrix}\,.
\end{array}\]
Remark that for any $\mathbf X\in \mathcal{S}_2$, 
\[(\left< \mathbf A_i,\mathbf X\right>=b_i\,,\,i\in [5])\Rightarrow\trace(\mathbf X)=4\,.\]
\end{example}
Next, we present an alternative iterative method, stated in Algorithm \ref{alg:sol.nonsmooth.hier}, to solve \eqref{eq:POP.on.variety}, based on  nonsmooth optimization methods, e.g., LMBM.  
It performs  well in practice for most cases and with significantly lower computational cost
when compared to the (currently fastest) SDP solver Mosek 9.1.

%

    \begin{algorithm}
    \footnotesize
    \caption{SpectralPOP-CTP}
    \label{alg:sol.nonsmooth.hier} 
    \textbf{Input:} POP \eqref{eq:POP.on.variety} with unknown optimal value $f^\star$ and optimal solutions;\\
    \hspace*{\algorithmicindent}\hspace*{\algorithmicindent} method (D) for  solving SDP with CTP. \\
    \textbf{Output:} increasing real sequence $(\tau_k)_{k\in\N}$ and $\x^\star\in\R^n$.
    \begin{algorithmic}[1]
    \For {$k\in\N$}{}
 \State Compute the optimal value $-\tau_k$ and an optimal solution $\mathbf X^\star$ of SDP \eqref{eq:SDP.form} by using method (D);
    \State Set $\mathbf M_k(\mathbf y^\star):=\mathbf P_{n,k}^{-1}\mathbf X^\star\mathbf P_{n,k}^{-1}$ (relying on \eqref{eq:convert.momentmat}) and extract an atom $\mathbf x^\star$ by using Henrion-Lasserre's algorithm in \cite{henrion2005detecting} from $\mathbf M_k(\mathbf y^\star)$;
    \State If $\mathbf x^\star$ exists, set $\tau_{k+j}=\tau_k$, $j\in\N^{\ge 1}$, and terminate.
    \EndFor
    \end{algorithmic}
    \end{algorithm}

Note that one can choose method (D) in Algorithm \ref{alg:sol.nonsmooth.hier} as Algorithm \ref{alg:sol.SDP.CTP.0} with LMBM solver or SketchyCGAL.
\begin{remark}
In practice, to  verify that an atom $\mathbf x^\star$ extracted in Step 3 of Algorithm \ref{alg:sol.nonsmooth.hier}  is an approximate optimal solution of POP \eqref{eq:POP.on.variety}, with given $\varepsilon\in (0,1)$, we check the following inequalities:
\[|f(\mathbf x^\star)-\tau_k|\le \varepsilon \|f\|_{\max}\text{ and }|h_j(\mathbf x^\star)|\le \varepsilon \|h_j\|_{\max}\,,\, j\in [l_g]\,,\]
where $\|p\|_{\max}:=\max_\alpha |p_\alpha|$ for any $p\in\R[x]$. 
We take $\varepsilon=0.01$ for the experiments in Section \ref{sec:benchmark}.
\end{remark}
Following Proposition \ref{pro:flatness.ball}, Corollary \ref{coro:well-defined.sdp.by.nonsmooth} and Proposition \ref{prop:strong.duality.attainability.POP.sphere}, we obtain the following corollary:
\begin{corollary}\label{coro:algorithm.well-definite}
(i) Sequence $(\tau_k)_{k\in\N}$ of Algorithm \ref{alg:sol.nonsmooth.hier} is well defined and $\tau_k\uparrow f^\star$ as $k\to\infty$.\\ 
(ii) Assume that condition (a) or (b) of Proposition \ref{prop:strong.duality.attainability.POP.sphere}.3 holds.
If there exists an optimal solution $\mathbf y^\star$ of SDP \eqref{eq:moment.hierarchy} for some order $k\in\N$ such that the flat extension condition holds, $\x^\star$ exists at the $k$-th iteration of Algorithm \ref{alg:sol.nonsmooth.hier}. 
 In this case, Algorithm \ref{alg:sol.nonsmooth.hier} terminates at the $k$-th iteration, $\x^\star$ is an optimal  solution of POP \eqref{eq:POP.on.variety} and $f^\star=\tau_k$.
\end{corollary}
In Corollary \ref{coro:algorithm.well-definite}, the flat extension condition implies that the SOS problem \eqref{eq:sos.hierarchy} has an optimal solution (due to \cite[Theorem 3.4 (b)]{lasserre2001global} and $\tau_k=\rho_k$), so that SDP \eqref{eq:SDP.form.dual} has an optimal solution.
In this case, $\mathbf X^\star$ exists, which in turn implies the existence of $\mathbf x^\star$.
%

In the two following subsections, we consider POPs on general compact sets as stated in Section \ref{sec:pop.on.sphere.ball}.

\subsubsection{Constrained POPs with single inequality (ball) constraint}
\label{sec:POP.compact.case1}
Assume that $l_g=1$  and $g_1=R-\|\mathbf x\|_2^2$.
In this case, $g=\{R-\|\mathbf x\|_2^2\}$. 
Let us show that POP \eqref{eq:POP.on.variety.ball} can be reduced to an equality constrained POP on a sphere.
By adding one slack variable $x_{n+1}$, the inequality constraint $R-\|\mathbf x\|_2^2\ge 0$ can be rewritten as an equality constraint $R-\|\mathbf x\|_2^2-x_{i+n}^2=0$ and so
\begin{equation}\label{POP.convert.sphere.0}
f^\star:=\inf\{\,f(\mathbf x)\ : \ (\mathbf x,x_{n+1})\in V(\bar h)\}\,,
\end{equation}
where $\bar h:=h\cup\{R-\|\mathbf x\|_2^2-x_{n+1}^2\}\subset \R[\mathbf x,x_{n+1}]$.

Notice that:
\begin{itemize}
    \item If $\bar {\mathbf x}^\star=(\mathbf x^\star,x_{n+1}^\star)$ is an optimal solution of POP \eqref{POP.convert.sphere.0}, $\mathbf x^\star$ is an optimal solution of POP \eqref{eq:POP.on.variety.ball}.
    \item Conversely, if $\mathbf x^\star$ is an optimal solution of POP \eqref{eq:POP.on.variety.ball}, then 
    $\bar {\mathbf x}^\star:=\left({\mathbf x}^\star,\sqrt{R-\|\mathbf x^\star\|_2^2}\right)$
    is an optimal solution of POP \eqref{POP.convert.sphere.0}.
\end{itemize}

Let us define $\bar n:=n+1$ and $\bar{\mathbf x}:=(\mathbf x,x_{n+1})$ to ease notation. 
For every $k\in\N$, consider the order $k$ moment relaxation of \eqref{POP.convert.sphere.0}:
\begin{equation}\label{eq:moment.hierarchy.multiply.diagonal.convert.0}
    \begin{array}{rl}
\bar \tau_k = \inf \limits_{\mathbf y \in {\R^{\stirlingii [0.7]{\bar n} {2k}} }} & L_{\mathbf y}(f)\\
\qquad \text{s.t. }& y_0=1\,,\,\mathbf M_k(\mathbf y) \succeq 0\,,\\
&\mathbf M_{k - 1 }(( R-\|\bar {\mathbf x}\|_2^2)\;\mathbf y)   = 0\,,\\
&\mathbf M_{k - \lceil h_j \rceil }(h_j\;\mathbf y)   = 0\,,\,j\in[l_h]\,.\\
\end{array}
\end{equation}
The corresponding dual SOS problem indexed by $k\in\N$ reads:
\begin{equation}\label{eq:sos.hierarchy.convert.dense.0}
\bar \rho_k\,:=\,\sup \,\{\,\xi\in\R\ :\ f-\xi \in P_k(\bar h)\}\,,
\end{equation}
where $P_k(\bar h)$ is the  truncated preodering of all polynomials of the form
\[\sigma_0
     +\psi_{0}(R-\| \bar\x\|^2_2)+\sum_{j=1}^{l_h}\psi_{j}h_j  \,, \] 
     with $\sigma_0\in\Sigma[\bar \x]_k$,  
      $\psi_{0}\in\R[\bar \x]_{2(k-1)}$, and 
     $\psi_j\in\R[\bar \x]_{2(k-\lceil h_j\rceil)}$, $j\in[l_h]$.
     
     The following lemma will be used later on:
\begin{lemma}\label{lem:to.show.dual.attain.0}
If $f-f^\star\in Q_k(g,h)$ for some $k\in\N$ then $f-f^\star\in P_k(\bar h)$.
\end{lemma}
\begin{proof}
By assumption, there exist $\sigma_0\in\Sigma[\mathbf x]_k$, $\sigma_1\in\Sigma[\mathbf x]_{k-1}$, and $\psi_{j}\in\R[\mathbf x]_{2(k-\lceil h_j\rceil)}$, $j\in[l_h]$ such that
\[f-f^\star=\sigma_0+\sigma_1(R-\|\mathbf x\|_2^2)+\sum_{j=1}^{l_h}\psi_{j}h_j= \sigma_0+\sigma_1x_{n+1}^2+\sigma_1(R-\|\bar{\mathbf  x}\|_2^2)+\sum_{j=1}^{l_h}\psi_{j}h_j \,,\]
yielding the result.
\end{proof}

     The strong duality, primal attainability, and dual attainability for primal-dual \eqref{eq:moment.hierarchy.multiply.diagonal.convert.0}-\eqref{eq:sos.hierarchy.convert.dense.0} are guaranteed in the following proposition:
\begin{proposition}\label{prop:strong.duality.attainability.POP.ball}
Let $f^\star$ be as in \eqref{eq:POP.on.variety.ball} with $g=\{R-\|x\|_2^2\}$. Then:
\begin{enumerate}
    \item Strong duality holds for primal-dual \eqref{eq:moment.hierarchy.multiply.diagonal.convert.0}-\eqref{eq:sos.hierarchy.convert.dense.0}  for large enough $k\in\N$.
    \item SDP \eqref{eq:moment.hierarchy.multiply.diagonal.convert.0} has an optimal solution for large enough $k\in\N$.
    \item Assume that one of the following two conditions holds:
    \begin{enumerate}
    \item  $Q(g,h)$ is Archimedean, the ideal $\left<h\right>$
is real radical, and the second-order sufficiency condition S2 (Definition \ref{def-S2}) holds at every global minimizer of POP  \eqref{eq:POP.on.variety.ball};
\item  $V(h)$ is finite.
\end{enumerate}
 Then  SDP \eqref{eq:sos.hierarchy.convert.dense.0} has an optimal solution for large enough $k\in\N$. In this case, $\bar \tau_k=\bar \rho_k=f^\star$.
 \end{enumerate}
 \end{proposition}   
 
\begin{proof}
The first and second statement follow from Proposition \ref{prop:strong.duality.ball}, after replacing $S(g,h)$ by $V(\bar h)$.
The third statement is due to Proposition \ref{prop:finite.conver.ball} and Lemma \ref{lem:to.show.dual.attain.0}.
\end{proof}

For every $k\in\N$, according to Lemma \ref{lem:trace.constant.property}, if $\mathbf M_{k - 1 }((R-\|\bar {\mathbf x}\|_2^2)\;\mathbf y)   = 0$ and $y_0=1$, then one has
\begin{equation}\label{eq:convert.momentmat.ball.0}
\trace(\mathbf P_{\bar n,k}\mathbf M_k(\mathbf y)\mathbf P_{\bar n,k})=(R+1)^k\,,
\end{equation}
where $\mathbf P_{\bar n,k}$ is defined as in \eqref{eq:P.mat} after replacing $n$ by $\bar n$. 
Thus SDP \eqref{eq:moment.hierarchy.multiply.diagonal.convert.0} has the CTP. 
We now do a similar process as in Section \ref{sec:POP.sphere}. 

Next, we present an iterative method, stated in Algorithm \ref{alg:sol.nonsmooth.hier.on.unique.ball}, to solve   \eqref{eq:POP.on.variety.ball} with $g=\{R-\|x\|_2^2\}$, based on a nonsmooth optimization method such as LMBM.  

%

    \begin{algorithm}
    \footnotesize
    \caption{SpectralPOP-CTP-WithSingleBallConstraint}
    \label{alg:sol.nonsmooth.hier.on.unique.ball} 
    \textbf{Input:} POP \eqref{eq:POP.on.variety.ball} with $g=\{R-\|x\|_2^2\}$, unknown optimal value $f^\star$ and optimal solutions;\\
    \hspace*{\algorithmicindent}\hspace*{\algorithmicindent} method (D) for  solving SDP with CTP. \\
    \textbf{Output:} increasing real sequence $(\bar \tau_k)_{k\in\N}$ and $\x^\star\in\R^n$.
    \begin{algorithmic}[1]
    \For {$k\in\N$}{}
 \State Compute the optimal value $-\bar \tau_k$ and an optimal solution $\mathbf y^\star$ of SDP \eqref{eq:moment.hierarchy.multiply.diagonal.convert.0} with CTP \eqref{eq:convert.momentmat.ball.0} by using method (D);
    \State Extract an atom $\bar{\mathbf x}^\star=(\mathbf x^\star,x_{n+1}^\star)$ by using Henrion-Lasserre's algorithm in \cite{henrion2005detecting} from $\mathbf M_k(\mathbf y^\star)$;
    \State If $\bar{\mathbf x}^\star$ exists, set $\bar \tau_{k+j}=\bar \tau_k$, $j\in\N^{\ge 1}$, and terminate.
    \EndFor
    \end{algorithmic}
    \end{algorithm}

Note that one can choose method (D) in Algorithm \ref{alg:sol.nonsmooth.hier.on.unique.ball}  as Algorithm \ref{alg:sol.SDP.CTP.0} with LMBM solver or SketchyCGAL.

Following Proposition \ref{pro:flatness.ball}, Corollary \ref{coro:well-defined.sdp.by.nonsmooth} and Proposition \ref{prop:strong.duality.attainability.POP.ball}, we obtain the following corollary:
\begin{corollary}\label{coro:algorithm.well-definite.unique.ball}
(i) Sequence $(\bar \tau_k)_{k\in\N}$ of Algorithm \ref{alg:sol.nonsmooth.hier.on.unique.ball} is well defined and $\bar \tau_k\uparrow f^\star$ as $k\to\infty$.\\ 
(ii) Assume that condition (a) or (b) of Proposition \ref{prop:strong.duality.attainability.POP.ball}.3 holds.
If there exists an optimal solution $\mathbf y^\star$ of SDP \eqref{eq:moment.hierarchy.multiply.diagonal.convert.0} for some order $k\in\N$ such that the flat extension condition holds, $\x^\star$ exists at the $k$-th iteration of Algorithm \ref{alg:sol.nonsmooth.hier.on.unique.ball}. 
 In this case, Algorithm \ref{alg:sol.nonsmooth.hier.on.unique.ball} terminates at the $k$-th iteration, $\x^\star$ is an optimal  solution of POP \eqref{eq:POP.on.variety.ball} and $f^\star=\bar \tau_k$.
\end{corollary}

%

\subsubsection{Constrained POPs on a ball}
\label{sec:POP.compact.case2}
Assume that $l_g> 1$  and $g_1=R-\|\mathbf x\|_2^2$.
Let us show that POP \eqref{eq:POP.on.variety.ball} can be reduced to an equality constrained POP on a sphere.
After adding $l_g$ slack variables $x_{n+i}$, $i\in[l_g]$, every inequality constraint $g_i(x)\ge 0$ can be rewritten as an equality constraint $g_i(x)=x_{i+n}^2$ and so
\[f^\star:=\inf\{\,f(\mathbf x)\ : \ (\mathbf x,x_{n+1},\dots,x_{n+l_g})\in V(\hat h)\}\,,\]
where $\hat h:=h\cup\{g_i-x_{i+n}^2: i\in [l_g]\}\subset \R[\mathbf x,x_{n+1},\dots,x_{n+l_g}]$.

Let us take upper bounds $b_i\ge \sup \{g_i(x): x\in S(\{g_1\},h)\}$, $i\in[l_g]$. 
For every $i\in [l_g]$, the bound $b_i$ can be computed by solving the order $k$ moment relaxation:
\begin{equation}\label{eq:bound.inequality.constrants}
    \begin{array}{rl}
-b_i = \inf \limits_{\mathbf y \in {\R^{\stirlingii [0.7]{n+1} {2k}} }} & L_{\mathbf y}(-g_i)\\
\qquad \text{s.t. }& y_0=1\,,\,\mathbf M_k(\mathbf y) \succeq 0\,,\\
&\mathbf M_{k - 1 }(( R-\|({\mathbf x},x_{n+1})\|_2^2)\;\mathbf y)   = 0\,,\\
&\mathbf M_{k - \lceil h_j \rceil }(h_j\;\mathbf y)   = 0\,,\,j\in[l_h]\,,\\
\end{array}
\end{equation}
based on the spectral minimization method presented in the previous section.

For every $(\mathbf x,x_{n+1},\dots,x_{n+l_g})\in V(\hat h)$, $\x\in S(g,h)$ and 
$x_{n+i}^2=g_i(\mathbf x)\le b_i\,,\,i\in[l_g]$,
since $S(g,h)\subset S(\{g_1\},h)$.
Therefore
\begin{equation}
\label{eq:create.ball.constraint}
    \|\mathbf x\|_2^2+\sum_{i=1}^{l_g}  x_{n+i}^2\,\leq\,\bar R \quad \text{with}\quad \bar R:=R+\sum_{i=1}^{l_g} b_i\,.
\end{equation}
Equivalently $V(\hat h)\subset B_{\bar R}^{n+l_g}$ and after adding 
one more slack variable $\x_{n+l_g+1}$:
\begin{equation}\label{POP.convert.sphere}
    f^\star:=\inf\{f(\mathbf x)\ : \ \bar {\mathbf x}\in V(\bar h)\}\,,
\end{equation}
 where $\bar {\mathbf x}:=(\mathbf x,x_{n+1},\dots,x_{n+l_g+1})$ and 
\[\bar h:=\hat h \cup \{\bar R-\|\bar {\mathbf x}\|_2^2\}\,=\,h\cup\{g_i-x_{i+n}^2: i\in [l_g]\}\cup \{\bar R-\|\bar {\mathbf x}\|_2^2\}\,\subset \R[\bar{\mathbf x}]\,.\]
Notice that:
\begin{itemize}
    \item If $\bar {\mathbf x}^\star=(\mathbf x^\star,x_{n+1}^\star,\dots,x_{n+l_g+1}^\star)$ is an optimal solution of POP \eqref{POP.convert.sphere}, $\mathbf x^\star$ is an optimal solution of POP \eqref{eq:POP.on.variety.ball}.
    \item Conversely, if $\mathbf x^\star$ is an optimal solution of POP \eqref{eq:POP.on.variety.ball}, then 
    \[\bar {\mathbf x}^\star:=\left({\mathbf x}^\star,\sqrt{g_1({\mathbf x}^\star)},\dots,\sqrt{g_{l_g}({\mathbf x}^\star)},\sqrt{\bar R-\sum_{i=1}^{l_g}g_i({\mathbf x}^\star)-\|\mathbf x^\star\|_2^2}\right)\]
    is an optimal solution of POP \eqref{POP.convert.sphere}.
\end{itemize}

Note $\bar n:=n+l_g+1$ for simplicity. 
For every $k\in\N$, consider the order $k$ moment relaxation of \eqref{POP.convert.sphere}:
\begin{equation}\label{eq:moment.hierarchy.multiply.diagonal.convert}
    \begin{array}{rl}
\bar \tau_k = \inf \limits_{\mathbf y \in {\R^{\stirlingii [0.7]{\bar n} {2k}} }} & L_{\mathbf y}(f)\\
\qquad \text{s.t. }& y_0=1\,,\,\mathbf M_k(\mathbf y) \succeq 0\,,\\
&\mathbf M_{k - \lceil g_i \rceil }((g_i-x_{n+i}^2)\;\mathbf y)   = 0\,,\,i\in[l_g]\,,\\
&\mathbf M_{k - 1 }((\bar R-\|\bar {\mathbf x}\|_2^2)\;\mathbf y)   = 0\,,\\
&\mathbf M_{k - \lceil h_j \rceil }(h_j\;\mathbf y)   = 0\,,\,j\in[l_h]\,.\\
\end{array}
\end{equation}
The corresponding dual SOS problem indexed by $k\in\N$ reads:
\begin{equation}\label{eq:sos.hierarchy.convert.dense}
\bar \rho_k\,:=\,\sup \,\{\,\xi\in\R\ :\ f-\xi \in P_k(\bar h)\}\,,
\end{equation}
where $P_k(\bar h)$ is the  truncated preodering of all polynomials of the form
\[\sigma_0+\sum_{i=1}^{l_g}\psi_i(g_i-x_{n+i}^2)
     +\psi_{l_g+1}(\bar R-\| \bar\x\|^2_2)+\sum_{j=1}^{l_h}\psi_{l_g+1+j}h_j\]
     with $\sigma_0\in\Sigma[\bar \x]_k$, 
     $\psi_i\in\R[\bar \x]_{2(k-\lceil g_i\rceil)}\,,\,i\in[l_g]$, 
      $\psi_{l_g+1}\in\R[\bar \x]_{2(k-1)}$, and 
     $\psi_{l_g+1+j}\in\R[\bar \x]_{2(k-\lceil h_j\rceil)}$, $j\in[l_h]$.

We will use  the following lemma later on:
\begin{lemma}\label{lem:to.show.dual.attain}
If $f-f^\star\in Q_k(g,h)$ for some $k\in\N$ then $f-f^\star\in P_k(\bar h)$.
\end{lemma}
\begin{proof}
By assumption, there exist $\sigma_0\in\Sigma[x]_k$, $\sigma_i\in\Sigma[x]_{k-\lceil g_i\rceil}$, $i\in[l_g]$, and $\psi_{j}\in\R[x]_{2(k-\lceil h_j\rceil)}$, $j\in[l_h]$ such that
\[f-f^\star=\sigma_0+\sum_{i=1}^{l_g}\sigma_ig_i+\sum_{j=1}^{l_h}\psi_{j}h_j\,.\]
It implies that 
\[f-f^\star= \sigma_0+\sum_{i=1}^{l_g}\sigma_ix_{i+n}^2+\sum_{i=1}^{l_g}\sigma_i(g_i-x_{i+n}^2)+0\times (\bar R-\|\bar {\mathbf x}\|_2^2)+\sum_{j=1}^{l_h}\psi_{j}h_j \,,\]
yielding the result.
\end{proof}

     The strong duality, primal attainability, and dual attainability for primal-dual \eqref{eq:moment.hierarchy.multiply.diagonal.convert}-\eqref{eq:sos.hierarchy.convert.dense} are guaranteed in the following proposition:
\begin{proposition}\label{prop:strong.duality.attainability.POP.ball.general}
Let $f^\star$ be as in \eqref{eq:POP.on.variety.ball}. Then:
\begin{enumerate}
    \item Strong duality holds for primal-dual \eqref{eq:moment.hierarchy.multiply.diagonal.convert}-\eqref{eq:sos.hierarchy.convert.dense}  for large enough $k\in\N$.
    \item SDP \eqref{eq:moment.hierarchy.multiply.diagonal.convert} has an optimal solution for large enough $k\in\N$.
    \item Assume one of the following two conditions holds:
    \begin{enumerate}
    \item  $Q(g,h)$ is Archimedean, the ideal $\left<h\right>$
is real radical, and the second-order sufficiency condition S2 (Definition \ref{def-S2}) holds at every global minimizer of POP  \eqref{eq:POP.on.variety.ball};
\item  $V(h)$ is finite.
\end{enumerate}
 Then  SDP \eqref{eq:sos.hierarchy.convert.dense} has an optimal solution for large enough $k\in\N$. In this case, $\bar \tau_k=\bar \rho_k=f^\star$.
 \end{enumerate}
 \end{proposition}   
 
\begin{proof}
The first and second statement follow from to Proposition \ref{prop:strong.duality.ball} after replacing $S(g,h)$ by $V(\bar h)$.
The third statement is due to Proposition \ref{prop:finite.conver.ball} and Lemma \ref{lem:to.show.dual.attain}.
\end{proof}

For every $k\in\N$, according to Lemma \ref{lem:trace.constant.property}, if $\mathbf M_{k - 1 }((\bar R-\|\bar {\mathbf x}\|_2^2)\;\mathbf y)   = 0$ and $y_0=1$,
\begin{equation}\label{eq:convert.momentmat.ball}
\trace(\mathbf P_{\bar n,k}\mathbf M_k(\mathbf y)\mathbf P_{\bar n,k})=(\bar R+1)^k\,,
\end{equation}
where $\mathbf P_{\bar n,k}$ is defined as in \eqref{eq:P.mat} with $n$ replaced by $\bar n$. 
Thus SDP \eqref{eq:moment.hierarchy.multiply.diagonal.convert} has the CTP. 
It remains to follow a process which is similar to the one from Section \ref{sec:POP.sphere}. 

Next, we present an iterative method, stated in Algorithm \ref{alg:sol.nonsmooth.hier.on.general.ball}, to solve POP  \eqref{eq:POP.on.variety.ball} with $g=R-\|x\|_2^2$, based on nonsmooth optimization methods such as LMBM.  

%

    \begin{algorithm}
    \footnotesize
    \caption{SpectralPOP-CTP-WithBallConstraint}
    \label{alg:sol.nonsmooth.hier.on.general.ball} 
    \textbf{Input:} POP \eqref{eq:POP.on.variety.ball} with $g_1=R-\|x\|_2^2$, unknown optimal value $f^\star$ and optimal solutions;\\
    \hspace*{\algorithmicindent}\hspace*{\algorithmicindent} method (D) for  solving SDP with CTP. \\
    \textbf{Output:} increasing real sequence $(\bar \tau_k)_{k\in\N}$ and $\x^\star\in\R^n$.
    \begin{algorithmic}[1]
    \For {$k\in\N$}{}
    \State Compute the optimal value $b_i$ of SDP \eqref{eq:bound.inequality.constrants} with CTP, $i\in[l_g]$,  by using method (D) and set $\bar R:=R+\sum_{i=1}^{l_g} b_i$;
 \State Compute the optimal value $-\bar \tau_k$ and an optimal solution $\mathbf y^\star$ of SDP \eqref{eq:moment.hierarchy.multiply.diagonal.convert} with CTP \eqref{eq:convert.momentmat.ball} by using method (D);
    \State Extract an atom $\bar {\mathbf x}^\star=(\mathbf x^\star,x_{n+1}^\star,\dots,x_{n+l_g+1}^\star)$ by using Henrion-Lasserre's algorithm in \cite{henrion2005detecting} from $\mathbf M_k(\mathbf y^\star)$;
    \State If $\bar{\mathbf x}^\star$ exists, set $\bar \tau_{k+j}=\bar \tau_k$, $j\in\N^{\ge 1}$, and terminate.
    \EndFor
    \end{algorithmic}
    \end{algorithm}

As in the single (ball) constraint case, one can choose method (D) in Algorithm \ref{alg:sol.nonsmooth.hier.on.general.ball}  as Algorithm \ref{alg:sol.SDP.CTP.0} with LMBM solver or SketchyCGAL.

Following Proposition \ref{pro:flatness.ball}, Corollary \ref{coro:well-defined.sdp.by.nonsmooth} and Proposition \ref{prop:strong.duality.attainability.POP.ball.general}, we obtain the following corollary:
\begin{corollary}\label{coro:algorithm.well-definite.general.ball}
(i) The sequence $(\bar \tau_k)_{k\in\N}$ of Algorithm \ref{alg:sol.nonsmooth.hier.on.general.ball} is well defined and $\bar \tau_k\uparrow f^\star$ as $k\to\infty$.\\ 
(ii) Assume that either condition (a) or condition (b) of Proposition \ref{prop:strong.duality.attainability.POP.ball.general}.3 holds.
If there exists an optimal solution $\mathbf y^\star$ of SDP \eqref{eq:moment.hierarchy.multiply.diagonal.convert} at order $k\in\N$ such that the flat extension condition holds, then $\x^\star$ exists at the $k$-th iteration of Algorithm \ref{alg:sol.nonsmooth.hier.on.general.ball} . 
 In this case, Algorithm \ref{alg:sol.nonsmooth.hier.on.general.ball}  terminates at the $k$-th iteration, $\x^\star$ is an optimal  solution of POP \eqref{eq:POP.on.variety.ball} and $f^\star=\bar \tau_k$.
\end{corollary}

\subsection{Systems of polynomial equations}
\label{subsec:sys.po.eq}
We suggest to use the adding spherical constraints (ASC)   method in \cite[Algorithm 4.3]{mai2019sums} to compute at least one real root of a system of polynomial equations. 
Let $V$ be a variety  contained in the unit sphere.
Let $\mathbf a_0=\mathbf 0$ and $(\mathbf a_1,\dots,\mathbf a_n)$ be the  canonical basis of $\R^n$.
In Algorithm \ref{alg:asc}, we recall the ASC algorithm to compute at least one feasible point of $V$:

\begin{algorithm}
\footnotesize
\caption{SpectralASC}
\label{alg:asc}
\textbf{Input:} variety $V$ contained in the unit sphere, relaxation order $k\in\N$.\\
\textbf{Output:} $\x^\star\in V$.
\begin{algorithmic}[1]
\For{$t\in[n]$}{}
\State Compute the optimal value $\omega_t$ and possible $\x^\star$ of POP $\min\{\|\mathbf x-\mathbf a_t\|_2^2\,:\,\mathbf x\in V\}$ by running the $k$-th iteration of Algorithm \ref{alg:sol.nonsmooth.hier};
\State If $\x^\star$ exists, terminate;
\State If $t\le n-1$, set $V=V\cap \{\mathbf x\in\R^n\,:\,\omega_t=\|\mathbf x-\mathbf a_t\|_2^2\}$;
\EndFor
\State Set $\x^\star=\mathbf a-\frac{1}2\omega$, with $\mathbf a = (1,\dots, 1 )\in\R^n$ and $\omega=(\omega_1,\dots,\omega_n)$.
\end{algorithmic}
\end{algorithm}

Consider a system of polynomial equations in the form 
\begin{equation}\label{eq:poly.sys.form}
    \mathbf x\in\R^n\quad\text{ and }\quad p_1(\mathbf x)=\dots=p_r(\mathbf x)=0\,,
\end{equation}
where $p_j\in\R[\mathbf x]$, $j\in [n]$. 
Assume that there exists a real root of \eqref{eq:poly.sys.form} belonging to $B^n_L=\{\mathbf x\in\R^n: L-\Vert \mathbf x\Vert_2^2\geq 0\}$ for some $L>0$.
By adding one variable $x_{n+1}$ and noting $p_{r+1}=L-\|\mathbf x\|_2^2-x_{n+1}^2$,  \eqref{eq:poly.sys.form} is equivalent to the system $p_1(\bar {\mathbf x})=\dots=p_{r+1}(\bar{\mathbf  x})=0$, where $\bar {\mathbf x}=(\mathbf x,x_{n+1})$. 
Set $\hat p_j=L^{-1}p_j(L^{1/2}\bar {\mathbf x})$, $j\in [r+1]$. 
Then \eqref{eq:poly.sys.form} is equivalent to the system $\hat p_1(\bar {\mathbf x})=\dots=\hat p_{r+1}(\bar {\mathbf x})=0$ with $\hat p_{r+1}=1-\|\bar {\mathbf x}\|_2^2$. 
We can now apply Algorithm \ref{alg:asc} to compute a real root of \eqref{eq:poly.sys.form} by finding a feasible point of the variety 
\begin{equation}
    \hat V=\{\bar {\mathbf x}\in\R^{n+1}\,:\, \hat p_1(\bar {\mathbf x})=\dots=\hat p_{r+1}(\bar {\mathbf x})=0\}\,.
\end{equation}
Note that if $\bar {\mathbf x}^\star=(\mathbf x^\star,x_{n+1}^\star) \in \hat V$, then $ {\mathbf x}^\star$ is a real root of \eqref{eq:poly.sys.form}.
Conversely, if $ {\mathbf x}^\star$ is a real root of \eqref{eq:poly.sys.form}, then $(\mathbf x^\star,\pm \sqrt{L-\|\mathbf x^\star\|_2^2}) \in \hat V$.
It implies that the number of real roots of \eqref{eq:poly.sys.form} belonging to $B^n_L$ is $|\hat V|/2$. 
Hence if the set of real roots of \eqref{eq:poly.sys.form} belonging to $B^n_L$ is finite, the variety $\hat V$ is finite.

\section{Numerical experiments}
\label{sec:benchmark}
Let us report numerical results obtained while relying on algorithms from Section \ref{sec:applica.SDP-CTP} to solve equality constrained QCQPs on a sphere, quartic minimization problems on the unit sphere and squared systems of polynomial equations. 

The experiments are performed in Julia 1.3.1 with the following packages:
\begin{itemize}
    \item SumOfSquare.jl \cite{weisser2019polynomial} is a modeling library to write and solve SDP relaxations of POPs, based on JuMP.jl and the SDP solver Mosek 9.1.
    \item LMBM.jl solves  unconstrained NSOPs with the limited-memory bundle method of Haarala et al. \cite{haarala2007globally,haarala2004new}.
    LMBM.jl calls Karmitsa's Fortran implementation of LMBM algorithm  \cite{karmitsa2007lmbm}.
    \item SketchyCGAL is a MATLAB package to handle SDP problems with CTP/BTP, implemented by Yurtsever et al.
   \cite{yurtsever2019scalable}. 
   We have implemented a Julia version (SketchyCGAL.jl) of SketchyCGAL to ensure fair comparison with LMBM.jl and SumOfSquare.jl.
      In this section,  SketchyCGAL is used as a solver for SDP \eqref{eq:SDP.form} in Algorithm \ref{alg:sol.nonsmooth.hier}  instead of Algorithm \ref{alg:sol.SDP.CTP.0} or \ref{alg:sol.SDP.CTP.0.btp}.
\end{itemize}
We also use the package Arpack.jl,  which is based on the implicitly restarted Lanczos's algorithm, to compute the largest eigenvalues and the corresponding eigenvectors of  real symmetric matrices of (potentially) large size. 

When POPs have equality constraints, SumOfSquare.jl uses reduced forms with Groebner basis instead of creating SOS multipliers, in order to reduce solving time.

The implementation of algorithms described in Section \ref{sec:applica.SDP-CTP} can be downloaded from the link: \href{https://github.com/maihoanganh/SpectralPOP}{https://github.com/maihoanganh/SpectralPOP}.

We use a desktop computer with an Intel(R) Core(TM) i7-8665U CPU @ 1.9GHz $\times$ 8 and 31.2 GB of RAM. 
The notation for our numerical results are given in Table \ref{tab:nontation}.
\begin{table}
    \caption{\small Notation}
    \label{tab:nontation}
\footnotesize
\begin{center}
\begin{tabular}{|m{2.1cm}|m{8cm}|}
\hline
$n$&the number of variables of the POP\\
\hline
$l_g$&the number of inequality constraints of the  POP\\
\hline
$l_h$&the number of equality constraints of the POP\\
\hline
$k$& the order of the moment-SOS relaxation or the  iteration of Algorithm \ref{alg:sol.nonsmooth.hier}\\
\hline
$s$&the  size of the positive semidefinite matrix involved in the SDP relaxation \\
\hline
$m$&the number of trace equality constraints of the SDP relaxation\\
\hline
SumOfSquares & SDP relaxation modeled by SumOfSquares.jl and solved by Mosek 9.1\\
\hline
CTP & the method described either in Section \ref{sec:POP.sphere}, Section \ref{sec:POP.compact.case1} or Section \ref{sec:POP.compact.case2} \\
\hline
BTP & the method described in Remark \ref{re:BTP.for.POP}\\
\hline
LMBM & SDP relaxation solved by spectral minimization, described in Section \ref{sec:sdp.ctp.btp} with the LMBM solver\\
\hline 
SketchyCGAL & SDP relaxation solved by SketchyCGAL\\
\hline
SpectralPOP & SDP relaxation handled by CTP or BTP method, with LMBM or SketchyCGAL solver\\
\hline
val& the optimal value of the SDP relaxation\\
\hline
gap& the relative optimality gap w.r.t. SumOfSquares, defined by
\[\text{gap}=\frac{|\text{val}-\text{val(SumOfSquares)}|}{|\text{val(SumOfSquares)}|}\]\\
\hline
$^*$& there exists at least one optimal solution of the POP, which can extracted by Henrion-Lasserre's algorithm in \cite{henrion2005detecting}\\
\hline
time & the total computation time of the SDP relaxation in seconds\\
\hline
$-$& the calculation did not finish in 3000 seconds or ran out of memory\\
\hline
\end{tabular}    
\end{center}
\end{table}
\subsection{Polynomial optimization}

\subsubsection{Random dense equality constrained QCQPs on the unit sphere}
\label{sec:experiment.random.QCQP.sphere}
\paragraph{Test problems:}
We construct several instances of POP \eqref{eq:POP.on.variety} as follows:
\begin{enumerate}
    \item Take $h_1=1-\|x\|_2^2$ and choose $f$, $h_j$, $j\in [l_h]\backslash\{1\}$ with degrees at most $2$;
    \item Each coefficient of the objective function $f$ is taken randomly in $(-1,1)$ with respect to the uniform distribution;
    \item Select a random point $\mathbf a\in \R^n$ in the unit sphere;
    \item For every $j\in [l_h]\backslash\{1\}$, all non-constant coefficients of $h_j$ are taken randomly in $(-1,1)$ with respect to the uniform distribution,  and the constant coefficient of $h_j$ is chosen such that $h_j(\mathbf a)=0$.
\end{enumerate}
By construction, $\mathbf a$ is a feasible solution. 
We use the method presented in Section \ref{sec:POP.sphere} (actually the $k$-th iteration of Algorithm \ref{alg:sol.nonsmooth.hier}) to solve these problems.
Numerical results are displayed in Table \ref{tab:exam1.d.equal.2} for the case $l_h=1$ and Table \ref{tab:random.qcqp.first.relax}, \ref{tab:random.qcqp} for the case $l_h=\lceil n/4\rceil$. 
For these results, we use the Julia version of SketchyCGAL, which runs much faster than the MATLAB version without compromising accuracy.

\begin{table}
    \caption{\small Numerical results for random dense equality constrained QCQPs on the unit sphere, described in Section \ref{sec:experiment.random.QCQP.sphere}, with $(l_g,l_h)=(0,1)$ and $k=1$.}
    \label{tab:exam1.d.equal.2}
\scriptsize
\begin{center}
    \begin{tabular}{|c|c|c|c|c|c|c|c|}
        \hline
        \multicolumn{1}{|c|}{\multirow{2}{*}{POP size}} &\multicolumn{2}{c|}{SumOfSquares}                                        & \multicolumn{4}{c|}{SpectralPOP (CTP)}                          \\ \cline{4-7}
        \multicolumn{1}{|c|}{} & \multicolumn{2}{c|}{(Mosek)}                                                                          & \multicolumn{2}{c|} {LMBM} & \multicolumn{2}{c|}{SketchyCGAL}                                                            \\ \hline
        $n$ & val& time& val& time&val& time
        \\ 
        \hline
 50 & -6.03407$^*$ &0.4& -6.03407$^*$ & 0.2 & -6.00885 &0.1 \\ \hline
75 & -6.80575$^*$ &3.0 & -6.80575$^*$ &0.3  & -6.63839 &0.2 \\ \hline
100 & -7.40739$^*$ & 12.9& -7.40739$^*$ & 0.6  & -7.33078 &1.0 \\ \hline
 125& -9.08461$^*$& 35.6 & -9.08461$^*$ & 0.8  & -9.01115&1.3  \\ \hline
150 & -9.10803$^*$& 85.5 & -9.10803$^*$& 1.3 & -9.01721 &1.5 \\ \hline
175 & -10.80922$^*$ & 156.7& -10.80922$^*$ & 1.7 & -10.67402&1.9 \\ \hline
200 & -10.73626$^*$& 367.7 & -10.73626$^*$ & 2.1  & -10.66782&3.7 \\ \hline
250 & -12.21817$^*$& 1362.3 & -12.21817$^*$& 4.8 &  -12.12735&6.3 \\ \hline
300 & -13.77690$^*$& 4039.2 & -13.77690$^*$ & 6.5  &-13.77146 &29.7 \\ \hline
350 & $-$& $-$ & -14.23574$^*$ & 13.8  &  -14.14768&18.8 \\ \hline
 400 & $-$ & $-$ & -16.78926$^*$& 16.5   & -16.54410&18.6 \\ \hline
 500 & $-$ & $-$& -18.72305$^*$& 47.8  &-18.72205&421.5  \\ \hline
 700 & $-$ & $-$& -20.75451$^*$ &126.3  & -20.59610&157.3 \\ \hline
 900 & $-$ & $-$& -24.39911$^*$ & 322.8&   -24.38234 & 571.6\\ \hline
 1200 & $-$ & $-$& -28.99977$^*$ & 697.6 & -28.93762  &752.1 \\ \hline
 1500 & $-$ & $-$& -32.09837$^*$  & 3561.9 &  -32.02957  & 3840.1\\ \hline
    \end{tabular}
    \end{center}
\end{table}

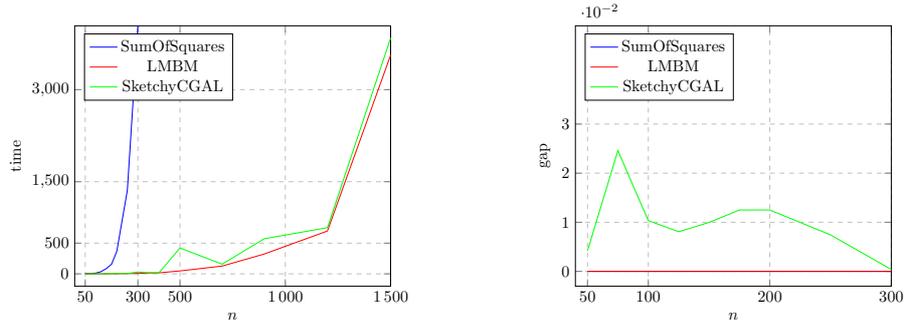
\begin{figure}
    \centering
    \subfigure{
    \begin{tikzpicture}[scale=\textwidth/20cm,samples=200]
\begin{axis}[
    log ticks with fixed point,
    x tick label style={/pgf/number format/1000 sep=\,},
    xlabel={$n$},
    ylabel={time},
    xmin=0, xmax=1500,
    ymin=-200, ymax=4039.2,
    xtick={50,300,500,1000,1500},
    ytick={0,500,1500,3000},
    legend pos=north west,
    ymajorgrids=true,
    xmajorgrids=true,
    grid style=dashed,
]
 
\addplot[
    color=blue
    ]
    coordinates {(50,0.4) (75,3.0) (100,12.9) (125,35.6) (150,85.5) (175,156.7) (200,367.7) (250,1362.3) (300,4039.2)};
    
\addplot[
    color=red,
    ]
    coordinates {(50,0.2) (75,0.3) (100,0.6) (125,0.8) (150,1.3) (175,1.7) (200,2.1) (250,4.8) (300,6.5) (350,13.8) (400,16.5) (500,47.8) (700,126.3) (900,322.8) (1200,697.6) (1500,3561.9)};
    
\addplot[
    color=green,
    ]
    coordinates {(50,0.1) (75,0.2) (100,1.0) (125,1.3) (150,1.5) (175,1.9) (200,3.7) (250,6.3) (300,29.7) (350,18.8) (400,18.6) (500,421.5) (700,157.3) (900,571.6) (1200,752.1) (1500,3840.1)};
    
    \legend{SumOfSquares, LMBM,SketchyCGAL}
 
\end{axis}
\end{tikzpicture}
}
\hfill
\subfigure{
\begin{tikzpicture}[scale=\textwidth/20cm,samples=200]
\begin{axis}[
    log ticks with fixed point,
    x tick label style={/pgf/number format/1000 sep=\,},
    xlabel={$n$},
    ylabel={gap},
    xmin=40, xmax=300,
    ymin=-0.003, ymax=0.05,
    xtick={50,100,200,300},
    ytick={0,0.01,0.02,0.03},
    legend pos=north west,
    ymajorgrids=true,
    xmajorgrids=true,
    grid style=dashed,
]
 
\addplot[
    color=blue
    ]
    coordinates {(50,0.0) (75,0.0) (100,0.0) (125,0.0) (150,0.0) (175,0.0) (200,0.0) (250,0.0) (300,0.0)};
    
\addplot[
    color=red,
    ]
    coordinates {(50,0.00001) (75,0.00001) (100,0.00001) (125,0.00001) (150,0.00001) (175,0.00001) (200,0.00001) (250,0.00001) (300,0.00001)};
    
\addplot[
    color=green,
    ]
    coordinates {(50,0.00419) (75,0.02459) (100,0.01034) (125,0.00808) (150,0.00997) (175,0.01250) (200,0.01250) (250,0.00743) (300,0.00039) };
    
    \legend{SumOfSquares, LMBM,SketchyCGAL}
 
\end{axis}
\end{tikzpicture}
}
    \caption{Efficiency and accuracy comparison for Table \ref{tab:exam1.d.equal.2}.}
    \label{fig:exam1.d.equal.2}
  \end{figure}


\begin{table}
    \caption{\small Numerical results for random dense  equality constrained QCQPs on the unit sphere, described in Section \ref{sec:experiment.random.QCQP.sphere}, with $(l_g,l_h)=(0,\lceil n/4\rceil)$ and $k=1$.}
    \label{tab:random.qcqp.first.relax}
\scriptsize
\begin{center}
\begin{tabular}{|c|c|c|c|c|c|c|c|c|}
        \hline
        \multicolumn{2}{|c|}{\multirow{2}{*}{POP size}}
        &\multicolumn{2}{c|}{SumOfSquares}                                        & \multicolumn{4}{c|}{SpectralPOP (CTP)}                          \\ \cline{5-8}
        \multicolumn{2}{|c|}{} & \multicolumn{2}{c|}{(Mosek)}                                                                          & \multicolumn{2}{c|} {LMBM} & \multicolumn{2}{c|}{SketchyCGAL}                                                            \\ \hline
$n$ & $l_h$ &\multicolumn{1}{c|}{val}& \multicolumn{1}{c|}{time}& \multicolumn{1}{c|}{val}& \multicolumn{1}{c|}{time}& \multicolumn{1}{c|}{val}& \multicolumn{1}{c|}{time}\\
\hline
50 & 14  & -4.80042$^*$&0.4 &-4.03646&0.6&-4.69448& 1.0\\\hline
60 & 16  & -3.95202&1.3 &-3.95202&0.9&-3.87651& 11.5\\\hline
70 & 19  & -6.14933&2.6 &-6.14933&1.1&-6.03721&4271.5\\\hline
80 & 21 &-6.20506$^*$&5.4 & -6.20506$^*$&1.8 &$-$&$-$\\\hline
100 & 26 &-6.58470&15.3 & -6.58470&3.6 &$-$&$-$\\\hline
120 & 31 &-6.96083&31.4 & -6.96083&7.8 &$-$&$-$\\\hline
150 & 39 &-6.92036&111.1 & -6.92036&17.8 &$-$&$-$\\\hline
200 & 51 &-10.13460&479.6 & -10.13460&70.0 &$-$&$-$\\\hline
300 & 76 &-11.86224&4761.1 & -11.86224&404.5 &$-$&$-$\\\hline
400 & 76 &$-$&$-$ & -13.28067&999.2 &$-$&$-$\\\hline
\end{tabular}    
\end{center}
\end{table}

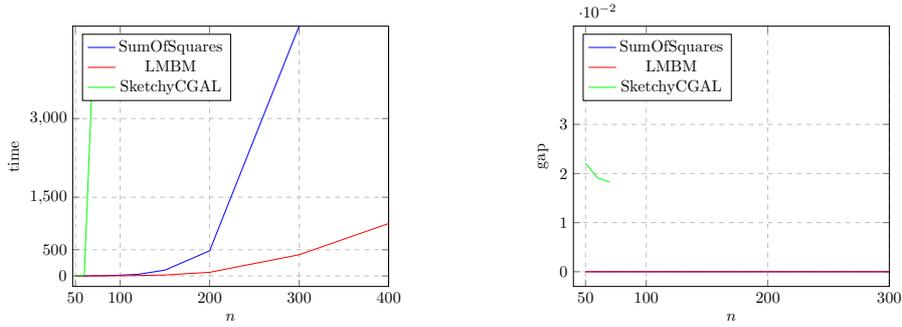
\begin{figure}
    \centering
    \subfigure{
    \begin{tikzpicture}[scale=\textwidth/20cm,samples=200]
\begin{axis}[
    log ticks with fixed point,
    x tick label style={/pgf/number format/1000 sep=\,},
    xlabel={$n$},
    ylabel={time},
    xmin=47, xmax=400,
    ymin=-200, ymax=4761.1,
    xtick={50,100,200,300,400},
    ytick={0,500,1500,3000},
    legend pos=north west,
    ymajorgrids=true,
    xmajorgrids=true,
    grid style=dashed,
]
 
\addplot[
    color=blue
    ]
    coordinates {(50,0.4) (60,1.3) (70,2.6) (80,5.4) (100,15.3) (120,31.4) (150,111.1) (200,479.6) (300,4761.1)};
    
\addplot[
    color=red,
    ]
    coordinates {(50,0.3) (60,0.9) (70,1.1) (80,1.8) (100,3.6) (120,7.8) (150,17.8) (200,70.0) (300,404.5) (400,999.2)};
    
\addplot[
    color=green,
    ]
    coordinates {(50,1.0) (60,11.5) (70,4271.5)};
    
    \legend{SumOfSquares, LMBM,SketchyCGAL}
 
\end{axis}
\end{tikzpicture}
}
\hfill
\subfigure{
\begin{tikzpicture}[scale=\textwidth/20cm,samples=200]
\begin{axis}[
    log ticks with fixed point,
    x tick label style={/pgf/number format/1000 sep=\,},
    xlabel={$n$},
    ylabel={gap},
    xmin=40, xmax=300,
    ymin=-0.003, ymax=0.05,
    xtick={50,100,200,300},
    ytick={0,0.01,0.02,0.03},
    legend pos=north west,
    ymajorgrids=true,
    xmajorgrids=true,
    grid style=dashed,
]
 
\addplot[
    color=blue
    ]
    coordinates {(50,0) (60,0) (70,0) (80,0) (100,0) (120,0) (150,0) (200,0) (300,0)};
    
\addplot[
    color=red,
    ]
    coordinates {(50,0) (60,0) (70,0) (80,0) (100,0) (120,0) (150,0) (200,0) (300,0)};
    
\addplot[
    color=green,
    ]
    coordinates {(50,0.02206) (60,0.01910) (70,0.01823)};
    
    \legend{SumOfSquares, LMBM,SketchyCGAL}
 
\end{axis}
\end{tikzpicture}
}
    \caption{Efficiency and accuracy comparison for Table \ref{tab:random.qcqp.first.relax}.}
    \label{fig:random.qcqp.first.relax}
  \end{figure}

\begin{table}
    \caption{\small Numerical results for random dense  equality constrained QCQPs on the unit sphere, described in Section \ref{sec:experiment.random.QCQP.sphere}, with $(l_g,l_h)=(0,\lceil n/4\rceil)$ and $k=2$.}
    \label{tab:random.qcqp}
\scriptsize
\begin{center}
\begin{tabular}{|c|c|c|c|c|c|c|c|c|c|c|}
        \hline
        \multicolumn{2}{|c|}{\multirow{2}{*}{POP size}}
        &\multicolumn{2}{c|}{\multirow{2}{*}{SDP size}}
        &\multicolumn{2}{c|}{SumOfSquares}                                        & \multicolumn{4}{c|}{SpectralPOP (CTP)}                          \\ \cline{7-10}
        \multicolumn{2}{|c|}{}& \multicolumn{2}{c|}{} & \multicolumn{2}{c|}{(Mosek)}                                                                          & \multicolumn{2}{c|} {LMBM} & \multicolumn{2}{c|}{SketchyCGAL}                                                            \\ \hline
$n$ & $l_h$&$s$&$m$ &\multicolumn{1}{c|}{val}& \multicolumn{1}{c|}{time}& \multicolumn{1}{c|}{val}& \multicolumn{1}{c|}{time}& \multicolumn{1}{c|}{val}& \multicolumn{1}{c|}{time}\\
\hline
5 & 2 & 21 & 148 & -2.32084$^*$&0.01 &-2.32084$^*$&0.2&-2.29957& 0.7\\\hline
10& 3 &66 & 1409 & -1.07536$^*$ &0.2 &-1.07536$^*$ &0.3 & -1.06480 &5.1\\\hline
15& 4 &136 & 5985 & -1.12894$^*$ & 5.6&-1.12894$^*$& 0.7 & -1.11512&55.2\\\hline
20& 5 & 231 & 17326 &-2.48514$^*$& 52.1&-2.48514$^*$& 2.2 & -2.46573&505.4\\\hline
25& 7 & 351 & 40483 &-2.80478$^*$& 460.8&-2.80478$^*$ & 16.2& -2.79507& 2127.2\\\hline
30& 8 &496 & 80849&-2.84989$^*$ & 3797.5&-2.84989$^*$ & 19.2& -2.83486& 2656.8\\\hline
35& 9 &666 & 145855&$-$ & $-$&-4.23210$^*$& 75.8  & $-$& $-$\\\hline
40& 10 &861& 243951 &$-$ & $-$&-4.49644$^*$ & 99.7 & $-$& $-$\\\hline
45& 12 &1081& 385918 &$-$ & $-$&-3.24527 & 256.8 & $-$& $-$\\\hline
50& 13 & 1326 &580789 &$-$& $-$&-4.16019& 351.9  & $-$& $-$\\\hline
55& 14 & 1596 &841625 &$-$& $-$&-3.71963& 799.5 &$-$& $-$\\\hline
60& 15 & 1891 &1181876 &$-$& $-$&-5.76124& 1800.1& $-$& $-$\\\hline
65& 15 & 2211 &1618453 &$-$& $-$&-4.61797 & 2714.4 & $-$& $-$\\\hline
\end{tabular}    
\end{center}
\end{table}

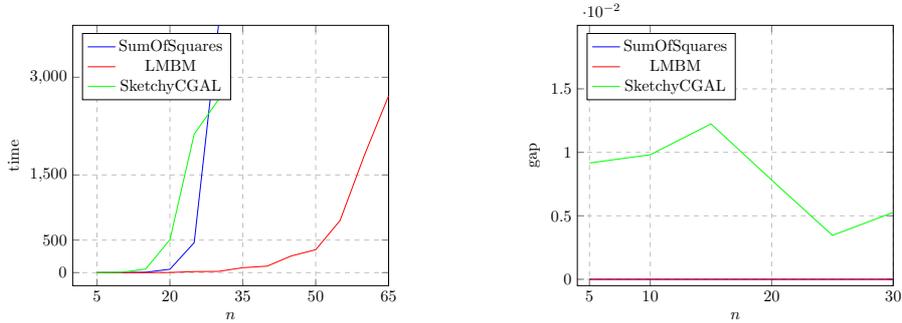
\begin{figure}
    \centering
    \subfigure{
    \begin{tikzpicture}[scale=\textwidth/20cm,samples=200]
\begin{axis}[
    log ticks with fixed point,
    x tick label style={/pgf/number format/1000 sep=\,},
    xlabel={$n$},
    ylabel={time},
    xmin=0, xmax=65,
    ymin=-200, ymax=3797.5,
    xtick={5,20,35,50,65},
    ytick={0,500,1500,3000},
    legend pos=north west,
    ymajorgrids=true,
    xmajorgrids=true,
    grid style=dashed,
]
 
\addplot[
    color=blue
    ]
    coordinates {(5,0.01) (10,0.2) (15,5.6) (20,52.1) (25,460.8) (30,3797.5)};
    
\addplot[
    color=red,
    ]
    coordinates {(5,0.2) (10,0.3) (15,0.7) (20,2.2) (25,16.2) (30,19.2) (35,75.8) (40,99.7) (45,256.8) (50,351.9) (55,799.5) (60,1800.1) (65,2714.4)};
    
\addplot[
    color=green,
    ]
    coordinates {(5,0.7) (10,5.1) (15,55.2) (20,505.4) (25,2127.2) (30,2656.8)};
    
    \legend{SumOfSquares, LMBM,SketchyCGAL}
 
\end{axis}
\end{tikzpicture}
}
\hfill
\subfigure{
\begin{tikzpicture}[scale=\textwidth/20cm,samples=200]
\begin{axis}[
    log ticks with fixed point,
    x tick label style={/pgf/number format/1000 sep=\,},
    xlabel={$n$},
    ylabel={gap},
    xmin=4, xmax=30,
    ymin=-0.0005, ymax=0.02,
    xtick={5,10,20,30},
    ytick={0,0.005,0.01,0.015},
    legend pos=north west,
    ymajorgrids=true,
    xmajorgrids=true,
    grid style=dashed,
]
 
\addplot[
    color=blue
    ]
    coordinates {(5,0.0) (10,0.0) (15,0.0) (20,0.0) (25,0.0) (30,0.0)};
    
\addplot[
    color=red,
    ]
    coordinates {(5,0.0) (10,0.0) (15,0.0) (20,0.0) (25,0.0) (30,0.0)};
    
\addplot[
    color=green,
    ]
    coordinates {(5,0.00916) (10,0.00981) (15,0.01224) (20,0.00781) (25,0.00346) (30,0.00527)};
    
    \legend{SumOfSquares, LMBM,SketchyCGAL}
 
\end{axis}
\end{tikzpicture}
}
    \caption{Efficiency and accuracy comparison for Table \ref{tab:random.qcqp}.}
    \label{fig:random.qcqp}
  \end{figure}

\paragraph{Efficiency comparison:}
In Table \ref{tab:exam1.d.equal.2}, we  minimize quadratic polynomials on the unit sphere. 
%
This relaxation for a POP in $n$ variables involves an SDP matrix of size $n+1$ and $2$ trace equality constraints. 
In this table, LMBM is the  fastest SDP solver while Mosek (the SDP solver used by SumOfSquares) is the  slowest. 
It is due to the fact that Mosek  relies on interior-point methods based on second order conditions to solve SDP while LMBM and SketchyCGAL only rely on algorithms based on first order conditions.
Note that we use the same modeling technique to generate the SDP-CTP relaxation solved with either SketchyCGAL or LMBM, so both related modeling times are the same. 
The solving time of SketchyCGAL is a bit smaller (resp. larger) than the one of LMBM when $n\le 400$ (resp. $n\ge 500$).

In Table \ref{tab:random.qcqp.first.relax} and Table  \ref{tab:random.qcqp}, we consider random equality constrained QCQPs and solve their first ($k=1$) and second ($k=2$) order moment relaxation, respectively.
In Table \ref{tab:random.qcqp.first.relax}, the size of the positive semidefinite matrix (resp. the number of trace equality constraints) involved in the SDP relaxation is equal to $n+1$ (resp. $l_h+1$).
In Table \ref{tab:random.qcqp}, the matrices involved in the SDP relaxation  have size $\stirlingii {n} {4}$ and the number of  trace equality constraints is $\mathcal{O}(\stirlingii {n} {4}^2)$, due to \eqref{eq:bound.mk}. 
Thus, the number of trace equality constraints for these SDP relaxations is more than 200 times larger than the matrix size, for almost all instance of Table \ref{tab:random.qcqp}. 
LMBM still happens to be the fastest solver in both Table \ref{tab:random.qcqp.first.relax} and Table  \ref{tab:random.qcqp}, but SumOfsquares is more  efficient than SketchyCGAL.
The most expensive step performed by Mosek (used by SumOfsquares) is to solve a  system of linear equations  coming from certain  complementarity conditions (see page 13 in \cite{dahl2012semidefinite} for more details). 
The linear system becomes harder to solve when the number of trace equality constraints is larger. 
This is in contrast with LMBM, which does not need to solve any such large size linear system of equations.
By comparison with LMBM, SketchyCGAL may perform a larger number of operations  \cite[Algorithm 6.1]{yurtsever2019scalable},  as emphasized later on.

\paragraph{Accuracy comparison:}
When $n\le 300$ in Table \ref{tab:exam1.d.equal.2},  $n\le 300$ in Table \ref{tab:random.qcqp.first.relax} or $n \le 20$ in Table \ref{tab:random.qcqp}, LMBM  converges to the exact optimal value of POPs with high accuracy, similarly to SumOfSquares. 
Both LMBM and SumOfSquares  can extract at least one approximate optimal solution by Henrion-Lasserre's algorithm  \cite{henrion2005detecting},  when $n\le 300$ in Table \ref{tab:exam1.d.equal.2} or $n \le 20$ in Table \ref{tab:random.qcqp}.
Moreover, LMBM can provide an approximate optimal solution even for large-scale problems with $n=1500$ in Table \ref{tab:exam1.d.equal.2} (resp. $n=40$ in Table  \ref{tab:random.qcqp}) and in several cases in Table \ref{tab:random.qcqp.first.relax}. Unfortunately SketchyCGAL cannot do the extraction procedure  successfully, because of its inaccurate output.

\paragraph{Storage and evaluation comparisons:} 
In Table \ref{tab:storage.comp} and \ref{tab:eval.comp}, we display some additional information related to Mosek, LMBM and SketchyCGAL, for the rows $n=5,10,15,20,25$ of  Table \ref{tab:random.qcqp}:
\begin{itemize}
    \item storage;
    \item $\# \mathcal{A}$: the number of evaluations of the linear operator $\mathcal{A}$ in SDP \eqref{eq:SDP.form.0};
    \item $\# \mathcal{A}^T$: the number of evaluations of the adjoint operator $\mathcal{A}^T$;
    \item $s_{\max}$: the largest size of symmetric matrices of which eigenvalues and eigenvectors are computed;
    \item $N_{\eig}$: the number of symmetric matrices of which eigenvalues and eigenvectors are computed.
\end{itemize}

\begin{table}
    \caption{\small Storage comparisons for the rows $n=5,10,15,20,25$ of  Table \ref{tab:random.qcqp}.}
    \label{tab:storage.comp}
\scriptsize
\begin{center}
\begin{tabular}{|c|c|c|c|}
      \hline
        \multicolumn{1}{|c|}{\multirow{2}{*}{}}
        &\multicolumn{1}{c|}{SumOfSquares}                                        & \multicolumn{2}{c|}{SpectralPOP (CTP)}                          \\ \cline{3-4}
        \multicolumn{1}{|c|}{}& \multicolumn{1}{c|}{(Mosek)}                                                                          & \multicolumn{1}{c|} {LMBM} & \multicolumn{1}{c|}{SketchyCGAL} \\
\hline
$n$& storage& storage&storage\\
\hline
5& 9.4 MB & 29 MB &1.1 GB\\
\hline
10& 91 MB & 69 MB &39 GB \\
\hline
15& 422 MB & 351 MB &320 GB \\
\hline
20& 1.3 GB & 1.2 GB &1.3 TB  \\
\hline
25& 3.5 GB & 4.2 GB &3.3 TB \\
\hline
\end{tabular}    
\end{center}
\end{table}

\begin{table}
    \caption{\small  Evaluation comparisons for the rows $n=5,10,15,20,25$ of  Table \ref{tab:random.qcqp}.}
    \label{tab:eval.comp}
\scriptsize
\begin{center}
\begin{tabular}{|c|c|c|c|c|c|c|c|c|}
      \hline
        \multicolumn{1}{|c|}{\multirow{2}{*}{}}
   & \multicolumn{8}{c|}{SpectralPOP (CTP)}                          \\ \cline{2-9}
        \multicolumn{1}{|c|}{}         &\multicolumn{4}{c|} {LMBM} & \multicolumn{4}{c|}{SketchyCGAL} \\
\hline
$n$&$\# \mathcal{A}$&$\# \mathcal{A}^T$&$s_{\max}$&$N_{\eig}$&$\# \mathcal{A}$&$\# \mathcal{A}^T$&$s_{\max}$&$N_{\eig}$\\
\hline
5&   21& 22&21 & 22 &1179 &18618 &18 & 1180\\
\hline
10 &32 &33 &66 &33 &1199 &25489 &25 &1200 \\
\hline
15& 840&841 &136 &841 &7699 & 294999& 47&7700 \\
\hline
20& 124 &125 &231 &125  &2492 &80467 &39 &2493 \\
\hline
25&9066 &9067 &351 & 9067& 2596 &90835 &42 & 2597\\
\hline
\end{tabular}    
\end{center}
\end{table}

Table \ref{tab:storage.comp} indicates that SumOfSquares requires a bit lower storage than LMBM only for the cases $n=5,25$. 
\if{LMBM storage is a bit larger than the one of SumOfSquares when $n=15,20$. }\fi
However, SketchyCGAL requires a  much larger storage than LMBM and SumOfSquares. 
It is due to the fact that SketchyCGAL performs a large number of evaluations of $\mathcal{A}$ and $\mathcal{A}^T$ while relying on three specific primitive computations (see \cite[Section 2.3]{yurtsever2019scalable}).
Compared to SketchyCGAL, LMBM performs a smaller number of evaluations. 
For instance, the number of  evaluations of LMBM is ten times smaller than the one of SketchyCGAL for the row  $n=25$ of Table \ref{tab:eval.comp}. 
Because of the large number $m$ of trace equality constraints, the evaluations of $\mathcal{A}$ and $\mathcal{A}^T$ in SDP relaxations of POPs is more expensive than the simple one related to the first order SDP relaxation of MAXCUT, which is solved  very efficiently by SketchyCGAL  (see \cite[Section 2.5]{yurtsever2019scalable}). 

These specific behaviors mainly come from the subroutines used by LMBM and SketchyCGAL to compute eigenvalues and eigenvectors.
While LMBM computes directly the largest eigenvector (and corresponding eigenvalue) of the matrix 
$\mathbf C-\mathcal{A}^T \mathbf z$ involved in the nonsmooth function from  \eqref{eq:func.phik.0}, SketchyCGAL computes indirectly the smallest eigenvalue of the matrix $\mathbf C + \mathcal{A}^T (\mathbf y+\beta (\mathbf z -\mathbf b))$ in Step 8 of \cite[Algorithm 6.1]{yurtsever2019scalable} while relying on the so-called ``ApproxMinEvec'' subroutine. 
When the ApproxMinEvec subroutine is implemented via \cite[Algorithm 4.2]{yurtsever2019scalable},  SketchyCGAL provides approximations of the smallest eigenvalue and eigenvector of each matrix $\mathbf C + \mathcal{A}^T (\mathbf y+\beta(\mathbf z -\mathbf b))$ by using the  randomized Lanczos method. 
It only requires to compute the smallest eigenvalue and eigenvector of a tridiagonal matrix of small size (e.g. $s_{\max}=42$ when $n=25$ in Table \ref{tab:eval.comp} while the value $s_{\max}$ of LMBM is 351).
Besides, SketchyCGAL  computes $\mathbf v_i^T (\mathbf C + \mathcal{A}^T (\mathbf y+\beta(\mathbf z -\mathbf b)) \mathbf v_i$ \footnote{the vector $\mathbf v_i$ is updated in Step 6 of  \cite[Algorithm 4.2]{yurtsever2019scalable}} within the loop from Step 5 of \cite[Algorithm 4.2]{yurtsever2019scalable} while relying on three primitive computations (see \cite[(2.4)]{yurtsever2019scalable} for more details), which yields a large number of evaluations of $\mathcal{A}^T$.
Because of its slow convergence, SketchyCGAL runs a larger number of iterations in Step 6 of \cite[Algorithm 6.1]{yurtsever2019scalable}.
Thus it computes a large number of evaluations of $\mathcal A$ in Step 9 of \cite[Algorithm 4.2]{yurtsever2019scalable},  e.g. $\# \mathcal A=2492$ when $n=20$ while the value $\# \mathcal A$ is 124 for LMBM.

Based on the above comparison, we emphasize that LMBM is  cheaper and faster than Mosek or SketchyCGAL while LMBM ensures the same accuracy as Mosek when solving SDP relaxations of equality constrained QCQPs on the unit sphere.

\subsubsection{Random dense QCQPs on the unit ball}
\label{sec:experiment.random.QCQP.ball}
\paragraph{Test problems:} We construct several samples of POP \eqref{eq:POP.on.variety.ball} as follows:
\begin{enumerate}
    \item Take $g_1=1-\|x\|_2^2$ and choose $f$, $g_i$, $i\in[l_g]\backslash\{1\}$, and $h_j$, $j\in [l_h]$ with degrees at most $2$;
    \item Each coefficient of the objective function $f$ is taken randomly in $(-1,1)$ with respect to the uniform distribution;
    \item Select a random point $\mathbf a\in \R^n$ in the unit ball, with respect to the uniform distribution;
    \item For each  $i\in[l_g]\backslash\{1\}$, all non-constant coefficients of $g_i$ are taken randomly in $(-1,1)$ with respect to the uniform distribution, and the constant coefficient of $g_i$ is chosen such that $g_i(\mathbf a)>0$;
    \item For $j\in[l_h]$, all non-constant coefficients of $h_j$ are taken randomly in $(-1,1)$ with respect to the uniform distribution, and the constant coefficient of $h_j$ is chosen such that  $h_j(\mathbf a)=0$.
\end{enumerate}
Numerical results are displayed in Table \ref{tab:exam:random.dense.qcqp.on.ball} for the case $(l_g,l_h)=(1,\lceil n/4\rceil)$ and Table \ref{tab:random.qcqp.2} for the case $(l_g,l_h)=(\lceil n/8\rceil,\lceil n/8\rceil)$. 
We recall the following notation:
\begin{itemize}
    \item CTP (LMBM): the SDP relaxation is solved via the method described in Section \ref{sec:POP.compact.case1} (the $k$-th iteration of Algorithm \ref{alg:sol.nonsmooth.hier.on.unique.ball}) or Section \ref{sec:POP.compact.case2} (the $k$-th iteration of Algorithm \ref{alg:sol.nonsmooth.hier.on.general.ball}) with the LMBM solver. 
    \item BTP (LMBM): the  SDP relaxation is solved via the method described in Remark \ref{re:BTP.for.POP} with the LMBM solver  (Algorithm \ref{alg:sol.SDP.CTP.0.btp}).

\end{itemize}
In Table \ref{tab:exam:random.dense.qcqp.on.ball} and Table \ref{tab:random.qcqp.2}, SumOfSquares and BTP solve relaxations involving matrices with the same  size, corresponding exactly to the size of the moment relaxation  \eqref{eq:moment.hierarchy.ball}.
\begin{table}
    \caption{\small Numerical results of random dense QCQPs on the unit ball, described in Section \ref{sec:experiment.random.QCQP.ball}, with $(l_g,l_h)=(1,\lceil n/4\rceil)$, and $k=2$.}
    \label{tab:exam:random.dense.qcqp.on.ball}
\scriptsize
\begin{center}
   \begin{tabular}{|c|c|c|c|c|c|c|c|c|c|c|}
        \hline
        \multicolumn{2}{|c|}{\multirow{2}{*}{POP size}}
        &\multicolumn{2}{c|}{SDP size}
        &\multicolumn{2}{c|}{SumOfSquares}                                        & \multicolumn{4}{c|}{SpectralPOP}                          \\ \cline{7-10}
        \multicolumn{2}{|c|}{}& \multicolumn{2}{c|}{(CTP)} & \multicolumn{2}{c|}{(Mosek)}                                                                          & \multicolumn{2}{c|} {CTP (LMBM)} & \multicolumn{2}{c|}{BTP (LMBM)}                                                            \\ \hline
$n$ & $l_h$&$s$&$m$ &
\multicolumn{1}{c|}{val}& \multicolumn{1}{c|}{time}& \multicolumn{1}{c|}{val}& \multicolumn{1}{c|}{time}& \multicolumn{1}{c|}{val}& \multicolumn{1}{c|}{time}\\
\hline
 5& 2 & 28 &281 &-2.37513$^*$ &0.03 &-2.37513$^*$&0.2&-3.43291& 6.8\\
\hline
 10& 3 & 78 &2029& -2.31074$^*$&0.2 &-2.31074$^*$&0.4&-2.89248&18.6 \\
\hline
 15& 4 & 153 &7702& -2.32752$^*$& 5.3&-2.32752$^*$&0.7&-3.26317&396.5 \\
\hline
 20& 5 & 253 &21000& -3.52091$^*$&60.0 &-3.52091$^*$&1.7&-4.88156& 3226.2\\
\hline
25& 7 & 378 &47251& -4.35441$^*$& 460.4&-4.35441$^*$&7.1&$-$&$-$ \\
\hline
30& 8 & 528 &92049& -2.98326$^*$&3484.6 &-2.98326$^*$&28.0&$-$&$-$\\
\hline
35& 9 & 703 &163097&$-$ &$-$ &-4.09827&139.3&$-$&$-$\\
\hline
40& 11 & 903 &269095&$-$ &$-$ &-3.82947&181.9 &$-$&$-$\\
\hline
45& 12 & 1128 &421121&$-$ &$-$ &-4.12012&276.3&$-$&$-$ \\
\hline
50& 13 & 1378 &628369&$-$ &$-$ &-5.02577&3328.4&$-$&$-$ \\
\hline
\end{tabular}    
\end{center}
\end{table}

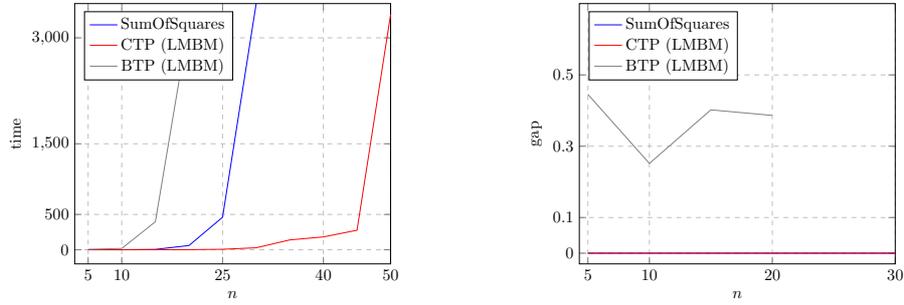
\begin{figure}
    \centering
    \subfigure{
    \begin{tikzpicture}[scale=\textwidth/20cm,samples=200]
\begin{axis}[
    log ticks with fixed point,
    x tick label style={/pgf/number format/1000 sep=\,},
    xlabel={$n$},
    ylabel={time},
    xmin=3, xmax=50,
    ymin=-200, ymax=3484.6,
    xtick={5,10,25,40,50},
    ytick={0,500,1500,3000},
    legend pos=north west,
    ymajorgrids=true,
    xmajorgrids=true,
    grid style=dashed,
]
 
\addplot[
    color=blue
    ]
    coordinates {(5,0.03) (10,0.2) (15,5.3) (20,60.0) (25,460.4) (30,3484.6)};
    
\addplot[
    color=red,
    ]
    coordinates {(5,0.2) (10,0.4) (15,0.7) (20,1.7) (25,7.1) (30,28.0) (35,139.3) (40,181.9) (45,276.3) (50,3328.4)};
    
\addplot[
    color=gray,
    ]
    coordinates {(5,6.8) (10,18.6) (15,396.5) (20,3226.2)};
    
    \legend{SumOfSquares, CTP (LMBM),BTP (LMBM)}
 
\end{axis}
\end{tikzpicture}
}
\hfill
\subfigure{
\begin{tikzpicture}[scale=\textwidth/20cm,samples=200]
\begin{axis}[
    log ticks with fixed point,
    x tick label style={/pgf/number format/1000 sep=\,},
    xlabel={$n$},
    ylabel={gap},
    xmin=4.3, xmax=30,
    ymin=-0.03, ymax=0.7,
    xtick={5,10,20,30},
    ytick={0,0.1,0.3,0.5},
    legend pos=north west,
    ymajorgrids=true,
    xmajorgrids=true,
    grid style=dashed,
]
 
\addplot[
    color=blue
    ]
    coordinates {(5,0) (10,0) (15,0) (20,0) (25,0) (30,0)};
    
\addplot[
    color=red,
    ]
    coordinates {(5,0) (10,0) (15,0) (20,0) (25,0) (30,0)};
    
\addplot[
    color=gray,
    ]
    coordinates {(5,0.44535) (10,0.25175) (15,0.40199) (20,0.38644)};
    
    \legend{SumOfSquares, CTP (LMBM),BTP (LMBM)}
 
\end{axis}
\end{tikzpicture}
}
    \caption{Efficiency and accuracy comparison for Table \ref{tab:exam:random.dense.qcqp.on.ball}.}
    \label{fig:exam:random.dense.qcqp.on.ball}
  \end{figure}

\begin{table}
    \caption{\small Numerical results of random dense QCQPs on the unit ball, described in Section \ref{sec:experiment.random.QCQP.ball}, with $(l_g,l_h)=(\lceil n/8\rceil,\lceil n/8\rceil)$, and $k=2$.}
    \label{tab:random.qcqp.2}
\scriptsize
\begin{center}
\begin{tabular}{|c|c|c|c|c|c|c|c|c|c|c|c|}
        \hline
        \multicolumn{3}{|c|}{\multirow{2}{*}{POP size}}
        &\multicolumn{2}{c|}{SDP size}
        &\multicolumn{2}{c|}{SumOfSquares}                                        & \multicolumn{4}{c|}{SpectralPOP}                          \\ \cline{8-11}
        \multicolumn{3}{|c|}{}& \multicolumn{2}{c|}{(CTP)} & \multicolumn{2}{c|}{(Mosek)}                                                                          & \multicolumn{2}{c|} {CTP (LMBM)} & \multicolumn{2}{c|}{BTP (LMBM)}                                                            \\ \hline
$n$ & $l_g$&$l_h$&$s$&$m$ &
\multicolumn{1}{c|}{val}& \multicolumn{1}{c|}{time}& \multicolumn{1}{c|}{val}& \multicolumn{1}{c|}{time}& \multicolumn{1}{c|}{val}& \multicolumn{1}{c|}{time}\\

 \hline
 10& 2 & 2& 105 &3711&-2.84974$^*$ &0.3 &-2.89467&12.9&-3.83990&6.2\\
\hline
 15& 2 & 2& 190 &11781&-3.49850$^*$ &6.5 &-3.50701&74.4&-4.70315&331.1\\
\hline
20& 3 & 3& 325 &34776&-2.17623$^*$& 161.8&-2.24255&191.6&-2.92872&7926.1\\
\hline
25& 4 & 4& 496 &81345&-3.55976$^*$ &1382.1 &-3.95982&975.4&$-$&$-$\\
\hline
30& 4 & 4& 666 &145855& -5.18136$^*$ &6605.8 &-5.41834&1118.3&$-$&$-$\\
\hline
35& 5 & 5& 903 &269095&$-$&$-$ &-5.30314&6983.6&$-$&$-$\\
\hline
\end{tabular}    
\end{center}
\end{table}

\begin{figure}
    \centering
    \subfigure{
    \begin{tikzpicture}[scale=\textwidth/20cm,samples=200]
\begin{axis}[
    log ticks with fixed point,
    x tick label style={/pgf/number format/1000 sep=\,},
    xlabel={$n$},
    ylabel={time},
    xmin=9.5, xmax=35,
    ymin=-200, ymax=6983.6,
    xtick={10,15,25,35},
    ytick={0,700,3000,6500},
    legend pos=north west,
    ymajorgrids=true,
    xmajorgrids=true,
    grid style=dashed,
]
 
\addplot[
    color=blue
    ]
    coordinates {(10,0.3) (15,6.5) (20,161.8) (25,1382.1) (30,6605.8)};
    
\addplot[
    color=red,
    ]
    coordinates {(10,12.9) (15,74.4) (20,191.6) (25,975.4) (30,1118.3) (35,6983.6)};
    
\addplot[
    color=gray,
    ]
    coordinates {(10,6.2) (15,331.1) (20,7926)};
    
    \legend{SumOfSquares,CTP (LMBM),BTP (LMBM)}
 
\end{axis}
\end{tikzpicture}
}
\hfill
\subfigure{
\begin{tikzpicture}[scale=\textwidth/20cm,samples=200]
\begin{axis}[
    log ticks with fixed point,
    x tick label style={/pgf/number format/1000 sep=\,},
    xlabel={$n$},
    ylabel={gap},
    xmin=9.5, xmax=30,
    ymin=-0.01, ymax=0.6,
    xtick={10,15,20,25,30},
    ytick={0,0.1,0.2,0.34},
    legend pos=north west,
    ymajorgrids=true,
    xmajorgrids=true,
    grid style=dashed,
]
 
\addplot[
    color=blue
    ]
    coordinates {(10,0) (15,0) (20,0) (25,0) (30,0)};
    
\addplot[
    color=red,
    ]
    coordinates {(10,0.01576) (15,0.00243) (20,0.03047) (25,0.11238) (30,0.04573)};
    
\addplot[
    color=gray,
    ]
    coordinates {(10,0.34745) (15,0.34433) (20,0.34577)};
    
    \legend{SumOfSquares, CTP (LMBM),BTP (LMBM)}
 
\end{axis}
\end{tikzpicture}
}
    \caption{Efficiency and accuracy comparison for Table \ref{tab:random.qcqp.2}.}
    \label{fig:random.qcqp.2}
  \end{figure}
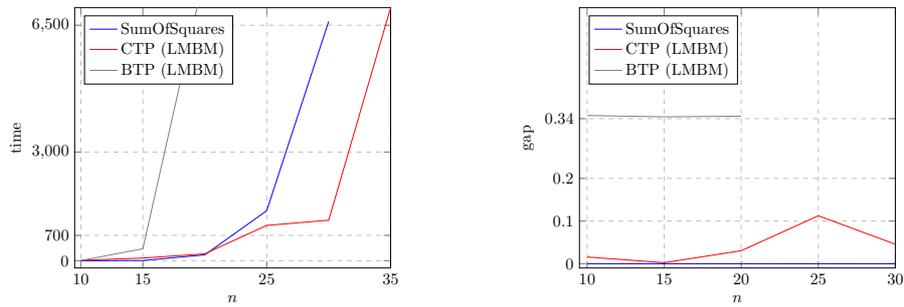

\paragraph{Efficiency and accuracy comparisons:} 
In Table \ref{tab:exam:random.dense.qcqp.on.ball}, we consider POPs which involve a single  inequality (ball) constraint.
In this case, CTP (LMBM) is the most  efficient and accurate solver.
Numerical results emphasize that SumOfSquares and CTP (LMBM) behave in a similar way as in Table \ref{tab:random.qcqp}.
This indicates that converting a POP with a single inequality (ball) constraint to a CTP-POP by adding one slack variable,  and solving the resulting SDP-CTP relaxation by means of spectral methods allows one to reduce the computing time while ensuring the same accuracy as the one obtained with SumOfSquares (Mosek). 
Note that when we use the method described in Section \ref{sec:POP.compact.case1}, the constant trace in \eqref{eq:convert.momentmat.ball.0} is always equal to $2^k$, which is independent of $n$. 

In Table \ref{tab:random.qcqp.2}, CTP (LMBM)  provides inaccurate output as it only yields lower bounds, while SumOfSquares still preserves accuracy.
Moreover, CTP (LMBM) is less (resp. more) efficient than SumOfSquares when $n\le 20$ (resp. $n\ge 25$). 
We also emphasize that when one relies on the method stated in Section \ref{sec:POP.compact.case2}, we obtain a value of ${\bar R}$, in \eqref{eq:create.ball.constraint}, for the sphere constraint of CTP-POP, which becomes larger when $n$ increases. 
It implies that the constant trace factor $(\bar R+1)^k$ in \eqref{eq:convert.momentmat.ball}  has a polynomial growth rate in $\bar R$. Thus we minimize a nonsmooth function of the form \eqref{eq:func.phik.0} with a large constant trace factor $a$.
The norm of the subgradient of this function at a point near its minimizers is rather large, which prevents LMBM to perform properly its minimization, by contrast with Table  \ref{tab:exam:random.dense.qcqp.on.ball}.
This difference of magnitude is shown in Table \ref{tab:norm.subgradients}, where we compute the subgradient norms during the last 10 iterations of CTP (LMBM) for the experiments from Table  \ref{tab:exam:random.dense.qcqp.on.ball} and Table \ref{tab:random.qcqp.2}   with $n=10$.

\begin{table}
    \caption{\small  Subgradient norms computed during the last 10 iterations of CTP (LMBM) for the experiments from Table  \ref{tab:exam:random.dense.qcqp.on.ball} and Table \ref{tab:random.qcqp.2} with $n=10$.}
    \label{tab:norm.subgradients}
\scriptsize
\begin{center}
\begin{tabular}{|c|c|c|c|c|c|c|c|c|c|c|c|c|c|c|}
\hline
Table  \ref{tab:exam:random.dense.qcqp.on.ball}  & 0.185& 0.098& 0.075& 0.097& 0.039& 0.019& 0.010& 0.007& 0.006& 0.002\\
\hline
Table \ref{tab:random.qcqp.2} & 65.9& 39.4& 48.0& 45.0& 37.6& 34.0& 33.9& 34.7& 26.9& 4.3 \\
\hline
\end{tabular}    
\end{center}
\end{table}

In both Table \ref{tab:exam:random.dense.qcqp.on.ball} and Table \ref{tab:random.qcqp.2}, BTP (LMBM) has the worst performance in terms on efficiency and accuracy. 
The trace bound  \eqref{eq:bound.trace.BTP.POP} obtained in Remark \ref{re:BTP.for.POP} is usually much larger than the ``exact'' trace of the optimal solution of the SDP relaxation. 
The same issue occurs for the subgradient norm of the nonsmooth function at a point near its minimizers.

According to our experience, LMBM is suitable for spectral minimization of SDP problems with trace bounds which are small enough and close to the exact trace value of the optimal solution. 
This seems to be the case for POPs with equality constraints and  few inequality constraints, and not for POPs with a significant number
of inequality constraints.

\subsubsection{Random dense quartics on the unit sphere}
\label{experiment:quartics.on.sphere}
\paragraph{Test problems:}
We construct several instances of POP \eqref{eq:POP.on.variety} as follows:
\begin{enumerate}
    \item Take $l_h=1$ and $h_1=1-\|x\|_2^2$ and choose $f$ with degree at most $4$;
    \item Each coefficient of the objective function $f$ is taken randomly in $(-1,1)$ with respect to the uniform distribution.
\end{enumerate}
We use the method presented in Section \ref{sec:POP.sphere} to solve these problems. 
The corresponding numerical results are displayed in Table \ref{tab:quartics.on.sphere}.

\begin{table}
    \caption{\small Numerical results for random dense quartics on the unit sphere, described in Section \ref{experiment:quartics.on.sphere}, with $(l_g,l_h)=(0,1)$ and $k=2$.}
    \label{tab:quartics.on.sphere}
\scriptsize
\begin{center}
    \begin{tabular}{|c|c|c|c|c|c|c|c|c|c|}
        \hline
        \multicolumn{1}{|c|}{\multirow{2}{*}{POP size}}&
        \multicolumn{2}{c|}{\multirow{2}{*}{SDP size}}
        &\multicolumn{2}{c|}{SumOfSquares}                                        & \multicolumn{4}{c|}{SpectralPOP (CTP)}                          \\ \cline{6-9}
        \multicolumn{1}{|c|}{} &
        \multicolumn{2}{c|}{} &
        \multicolumn{2}{c|}{(Mosek)}                                                                          & \multicolumn{2}{c|} {LMBM} & \multicolumn{2}{c|}{SketchyCGAL}                                                            \\ \hline
        $n$ &$s$&$m$& val& time& val& time&val& time
        \\ 
        \hline
5 &21&127& -2.74690$^*$ &0.02& -2.74699$^*$ & 0.3 & -2.72892 &0.3 \\ \hline
10 &66&1277& -3.63546$^*$ &0.4& -3.63585$^*$ & 0.6 & -3.62581 &2.9 \\ \hline
15 &136&5577& -4.06999$^*$ &7.6& -4.07015$^*$ & 2.0 & -4.06057 &34.5 \\ \hline
20 &231&16402& -3.94869$^*$ &83.2& -3.94913$^*$ & 47.2 & -3.94061 &249.9 \\ \hline
25 &351&38377& -4.23647$^*$ &652.4& -4.23699$^*$ & 306.3 & -4.22619 &508.8 \\ \hline
30 &496&77377& -4.24863$^*$ &5214.9& -4.247862 & 2358.9 & -4.23958 & 3323.5\\ \hline
    \end{tabular}
    \end{center}
\end{table}

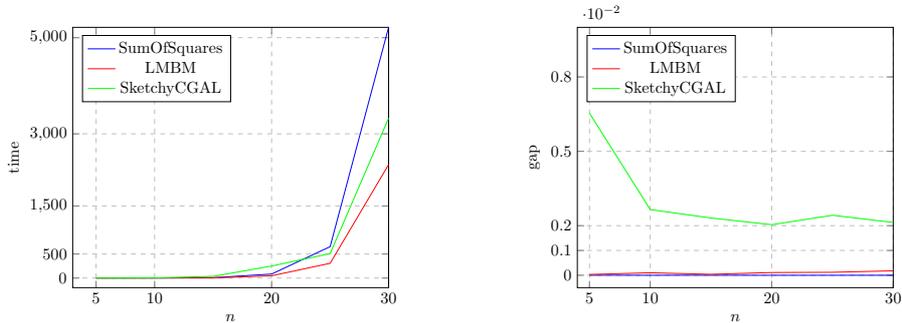
\begin{figure}
    \centering
    \subfigure{
    \begin{tikzpicture}[scale=\textwidth/20cm,samples=200]
\begin{axis}[
    log ticks with fixed point,
    x tick label style={/pgf/number format/1000 sep=\,},
    xlabel={$n$},
    ylabel={time},
    xmin=3, xmax=30,
    ymin=-200, ymax=5214.9,
    xtick={5,10,20,30},
    ytick={0,500,1500,3000,5000},
    legend pos=north west,
    ymajorgrids=true,
    xmajorgrids=true,
    grid style=dashed,
]
 
\addplot[
    color=blue
    ]
    coordinates {(5,0.02) (10,0.4) (15,7.6) (20,83.2) (25,652.4) (30,5214.9)};
    
\addplot[
    color=red,
    ]
    coordinates {(5,0.3) (10,0.6) (15,2.0) (20,47.2) (25,306.3) (30,2358.9)};
    
\addplot[
    color=green,
    ]
    coordinates {(5,0.3) (10,2.9) (15,34.5) (20,249.9) (25,508.8) (30,3323.5)};
    
    \legend{SumOfSquares, LMBM,SketchyCGAL}
 
\end{axis}
\end{tikzpicture}
}
\hfill
\subfigure{
\begin{tikzpicture}[scale=\textwidth/20cm,samples=200]
\begin{axis}[
    log ticks with fixed point,
    x tick label style={/pgf/number format/1000 sep=\,},
    xlabel={$n$},
    ylabel={gap},
    xmin=4, xmax=30,
    ymin=-0.0005, ymax=0.01,
    log basis x={10},
    xtick={5,10,20,30},
    ytick={0,0.001,0.002,0.005,0.008},
    legend pos=north west,
    ymajorgrids=true,
    xmajorgrids=true,
    grid style=dashed,
]
 
\addplot[
    color=blue
    ]
    coordinates {(5,0.0) (10,0.0) (15,0.0) (20,0.0) (25,0.0) (30,0.0)};
    
\addplot[
    color=red,
    ]
    coordinates {(5,3.2e-5) (10,0.00010) (15,3.9e-5) (20,0.00011) (25,0.00012) (30,0.00018)};
    
\addplot[
    color=green,
    ]
    coordinates {(5,0.00654) (10,0.00265) (15,0.00231) (20,0.00204) (25,0.00242) (30,0.00213)};
    
    \legend{SumOfSquares, LMBM,SketchyCGAL}
 
\end{axis}
\end{tikzpicture}
}
    \caption{Efficiency and accuracy comparison for Table \ref{tab:quartics.on.sphere}.}
    \label{fig:quartics.on.sphere}
  \end{figure}

\paragraph{Efficiency and accuracy comparisons:}
Table \ref{tab:quartics.on.sphere} indicates that LMBM is about twice faster than SumOfSquares when $n\ge 10$ as well as SketchyCGAL when $n\ge 25$.
While SketchyCGAL can be rather inaccurate, LMBM has an accuracy which is similar to SumOfSquares (Mosek), yielding the ability to extract optimal solutions of POPs when $n\le 30$.


\subsection{Squared systems of polynomial equations}
\label{sec:experiment.poly.sys}
Here we consider the problem of finding real roots of several squared systems of polynomial equations, issued from 
\href{http://homepages.math.uic.edu/~jan/demo.html}{the database of polynomial systems}
and \cite{didrit1997analyse} (the ``stewgou'' polynomial  system can be found in this later reference).
These systems have the form \eqref{eq:poly.sys.form} with $r=n$, namely  $p_1(\mathbf x)=\dots=p_n(\mathbf x)=0$.
We rely on the method described in Section \ref{subsec:sys.po.eq} and Algorithm \ref{alg:asc} to solve these systems.
Our numerical result are displayed in Table \ref{tab:poly.sys}, with the following notation:
\begin{itemize}
\item $n$: the number of variables and equations.
\item $d$: the maximal degree of the polynomials involved in the equations, i.e., $d=\max_j\deg(p_j)$;
\item $k$: the relaxation order given as input of Algorithm \ref{alg:asc};
\item $t$: the maximal number of iterations performed by Algorithm \ref{alg:asc};
\item $N$: the number of real roots obtained by using Algorithm \ref{alg:asc}.
\end{itemize}

\begin{table}
    \caption{\small Numerical results of squared systems of polynomial equations, described in Section \ref{sec:experiment.poly.sys}, with $L=10^4$.}
    \label{tab:poly.sys}
\scriptsize
\begin{center}
\begin{tabular}{|c|c|c|c|c|c|c|}
        \hline
        \multicolumn{1}{|c|}{}  &\multicolumn{2}{c|}{System size }&\multicolumn{1}{c|}{Order}   &\multicolumn{3}{c|}{SpectralASC: LMBM}\\
\hline
         & $n$ & $d$& $k$ & $t$  &time & $N$  \\
 \hline
katsura6 & 7 &2 &2& 1 & 3.1 & 2 \\
\hline
katsura7 & 8 &2 &2& 2 & 7.9 & 2 \\
\hline
katsura8 & 9 &2 &2& 1 & 4.4 & 2 \\
\hline
katsura9 & 10 &2 &2& 1 &  5.1 & 2 \\
\hline
katsura10 & 11 &2  &2& 1 &  8.7 & 2\\
\hline
stewgou & 9 &2 &2& 2 &  25.2 & 2\\
\hline
pole27sys & 14 &2 &1& 1 &  0.3 & 1\\
\hline
pole28sys & 16 &2 &1& 1 &  0.3 & 1\\
\hline
ku10 & 10 &2 &1& 1 &  0.3 & 1\\
\hline
chemkin & 10 &2 &2& 2 &  105.1 & 1\\
\hline
d1 & 12 &3 &2& 2 &  832.1 & 2\\
\hline
kin1 & 12 &3 &2& 1 &  611.7 & 2\\
\hline
i1 & 10 &3 &3& 1 &  133.4 & 1\\
\hline
\end{tabular}    
\end{center}
\end{table}
The total time required to solve each system by using our ASC algorithm together with the LMBM solver is less than 15 minutes, even for systems involving 16 variables.
This is in deep contrast with the recorded solving times of our original ASC algorithm  \cite[Algorithm 4.3]{mai2019sums}, which  can typically spends up to a hour to solve systems with 10 variables while relying on Mosek.
Because of the above mentioned accuracy issues, we could not use SketchyCGAL as a solver for Algorithm \ref{alg:asc}.
\section{Conclusion}
We have presented a nonsmooth hierarchy of SDP relaxations for optimizing polynomials on varieties contained in a Euclidean sphere.  
The advantage of this hierarchy is to circumvent the hard constraints involved in the standard SDP hierarchy \eqref{eq:moment.hierarchy}  by minimizing the maximal eigenvalue of a matrix pencil.
This in turn boils down to solving an unconstrained convex nonsmooth optimization problem by  LMBM and to computing  largest eigenvalues by means of the modified Lanczos's algorithm.
Our numerical experiments  indicate that solving this nonsmooth hierarchy is more efficient and more robust than solving the classical semidefinite hierarchy by interior-point methods, at least for a class of interesting POPs, including equality constrained QCQPs on the sphere, QCQPs with a single inequality (ball) constraint, and minimization of quartics on the sphere.
Our CTP framework can be further applied for an interesting class of noncommutative polynomial optimization, in particular for eigenvalue maximization problems arising from quantum information theory, where the variables are unitary operators \cite{navascues2008convergent}. 
A topic of future investigation is to handle in a more subtle way the case of POPs involving several inequalities. 
Our current method transforms such a POP into a  CTP-POP by adding a slack variable for each inequality.
One promising workaround would be to exploit the inherent sparse structure of this CTP-POP.
Another similar investigation track would be to exploit the CTP of SDP relaxations resulting from  polynomial optimization problems with sparse input data. 

Eventually, we have tried  to use spectral methods to solve SDP relaxations of QCQPs involving inequalities only, MAXCUT problems and 0/1 linear constrained quadratic problems. 
However, our preliminary experiments for these problems have not been  convincing in terms of  efficiency and accuracy.
In order to improve upon these results, one possible remedy would be to index the moment matrices by alternative Legendre/Chebychev bases, rather than with the standard monomial basis.
\paragraph{\textbf{Acknowledgements}.} 
The first author was supported by the MESRI funding from EDMITT.
The second author was supported by the FMJH Program PGMO (EPICS project) and  EDF, Thales, Orange et Criteo, as well as from the Tremplin ERC Stg Grant ANR-18-ERC2-0004-01 (T-COPS project).
This work has benefited from the Tremplin ERC Stg Grant ANR-18-ERC2-0004-01 (T-COPS project), the European Union's Horizon 2020 research and innovation programme under the Marie Sklodowska-Curie Actions, grant agreement 813211 (POEMA) as well as from the AI Interdisciplinary Institute ANITI funding, through the French ``Investing for the Future PIA3'' program under the Grant agreement n$^{\circ}$ANR-19-PI3A-0004.
The third author was supported by the European Research Council (ERC) under the European's Union Horizon 2020 research and innovation program (grant agreement 666981 TAMING).
\appendix

\section{Appendix}
\label{sec:Appendix}
\subsection{Spectral minimizations of SDP}
In this section, we provide the proofs of lemmas stated in Section \ref{sec:sdp.ctp} and \ref{sec:sdp.btp}.
First we recall the following useful properties of $\mathcal S$ and $\mathcal S^+$:
\begin{itemize}
    \item If $\mathbf X=\diag(\mathbf X_1,\dots,\mathbf X_l)\in \mathcal S$,
\begin{equation}\label{eq:pos.defi.diagon}
   \mathbf X\succeq 0\Longleftrightarrow \mathbf X_j\succeq 0\,,\,j\in [l]\qquad \text{ and }\qquad \trace(\mathbf X)=\sum_{j=1}^l \trace(\mathbf X_j)\,.
\end{equation}
\item If $\mathbf A=\diag(\mathbf A_1,\dots,\mathbf A_l)\in \mathcal S$ and $\mathbf B=\diag(\mathbf B_1,\dots,\mathbf B_l)\in \mathcal S$,
\begin{equation}\label{eq:scalar.prod.dia}
    \left<\mathbf A, \mathbf B\right> = \sum_{j=1}^l\left<\mathbf A_j, \mathbf B_j\right>\,.
\end{equation}
\end{itemize}
\subsubsection{SDP with Constant Trace Property}

\paragraph{Proof of Lemma \ref{lem:obtain.dual.sol}:}
\label{proof:lem:obtain.dual.sol}
\begin{proof}
The proof of \eqref{eq:nonsmooth.hierarchy.0}
is similar in spirit to the one of Helmberg and Rendl in \cite[Section 2]{helmberg2000spectral}.
Here, we extend this proof for SDP \eqref{eq:SDP.form.0}, which involves a block-diagonal positive semidefinite matrix.
From 
\eqref{eq:SDP.form.0},
\begin{equation*}
-\tau = \sup_{\mathbf X\in \mathcal{S}}\{ \left< \mathbf C,\mathbf X\right>\,:\,\mathcal{A} \mathbf X=\mathbf b\,,\,\trace(\mathbf X)=a\,,\, \mathbf X \succeq 0\}\,.
\end{equation*}
The dual of this SDP reads:
\begin{equation*}
-\rho = \inf_{(\mathbf z,\zeta)} \{ \mathbf b^T\mathbf z+a\zeta \,:\,
\mathcal{A}^T \mathbf z+\zeta\mathbf I-\mathbf C\succeq 0 \}\,,
\end{equation*}
where $\mathbf I$ is the identity matrix of size $s$.
From this,
\begin{equation*}
    \begin{array}{rl}
      -\rho&=  \inf _{(\mathbf z,\zeta)} \{ \mathbf b^T\mathbf z+a\zeta \,:\,
\zeta\ge \lambda_1(\mathbf C-\mathcal{A}^T \mathbf z)\} \\
         & =  \inf \{a\lambda_1(\mathbf C-\mathcal{A}^T \mathbf z)+ \mathbf b^T\mathbf z \,:\,\mathbf z\in\R^m \}\,. 
    \end{array}
\end{equation*}
Since $\rho=\tau$, \eqref{eq:nonsmooth.hierarchy.0} follows.
For the second statement, let $\mathbf z^\star$ be an optimal solution of SDP \eqref{eq:SDP.form.dual.0}. 
Then $\mathbf b^T\mathbf z^\star=-\rho=-\tau$. 
In addition, $\mathbf C-\mathcal{A}^T\mathbf z^\star \preceq 0$ implies that 
\[\lambda_1(\mathbf C-\mathcal{A}^T{\mathbf  z^\star})\le 0\,,\]
so that  $\varphi(\mathbf z^\star)\le -\tau$. 
Note that \eqref{eq:nonsmooth.hierarchy.0} indicates that $\varphi(\mathbf z^\star)\ge -\tau$. 
Thus, $\varphi(\mathbf z^\star)=-\tau$, yielding the second statement.
\end{proof}

The following proposition recalls the  differentiability properties of $\varphi$.
\begin{proposition}\label{prop:properties.phi.0}
The function $\varphi$ in \eqref{eq:func.phik.0} has the following properties:
\begin{enumerate}
\item  $\varphi$ is convex and continuous but not differentiable.
\item The subdifferential of $\varphi$ at $\mathbf z$ reads:
\begin{equation}\label{eq:subgrad.phi.fo.0}
    \partial \varphi(\mathbf z)=\{\mathbf b-a\mathcal{A}\mathbf W \ :\ \mathbf W\in \conv(\Gamma(\mathbf C-\mathcal{A}^T\mathbf z))\}\,,
\end{equation}
where for each $\mathbf A\in \mathcal{S}$, 
\begin{equation}\label{eq:Gamma}
    \Gamma(\mathbf A):=\{\mathbf u\mathbf u^T\ :\ \mathbf A\mathbf u=\lambda_1(\mathbf A)\mathbf u\ ,\ \|\mathbf u\|_2=1\}\,.
\end{equation}
\end{enumerate}
\end{proposition}
\begin{proof}
Properties 1-2 are from Helmberg-Rendl \cite[Section 2]{helmberg2000spectral} (see also \cite[(4)]{overton1992large}). 
\end{proof}
The following result is useful to recover an optimal solution of SDP \eqref{eq:SDP.form.0} from an optimal solution of NSOP \eqref{eq:nonsmooth.hierarchy.0}.
\begin{lemma}\label{lem:extract.sdp.solu}
 If $\mathbf{\bar z}$ is an optimal solution of NSOP \eqref{eq:nonsmooth.hierarchy.0},
then:
\begin{enumerate}
    \item There exists $\mathbf X^\star\in a\conv(\Gamma(\mathbf C-\mathcal{A}^T\mathbf{\bar z}))$ such that  $\mathcal{A}\mathbf X^\star=\mathbf b$.
    \item $\mathbf X^\star=a\sum_{j=1}^r \bar \xi_j\mathbf u_j\mathbf u_j^T$, where $\mathbf u_1,\dots,\mathbf u_r$ are all uniform eigenvectors  corresponding to $\lambda_1(\mathbf C-\mathcal{A}^T\mathbf{\bar z})$ and $(\bar \xi_1,\dots,\bar \xi_r)$ is an optimal solution of QP \eqref{eq:socp.sdp.sol}.
    \item $\mathbf X^\star$ is an optimal solution of SDP \eqref{eq:SDP.form.0}.
\end{enumerate}
 
\end{lemma}
\begin{proof}
By \cite[Theorem 4.2]{bagirov2014introduction}, $\mathbf 0\in \partial \varphi(\mathbf{\bar z})$.
Combining this with Proposition \ref{prop:properties.phi.0}.2, 
the first statement follows, which in turn implies  the second statement.
We next prove the third statement.
Since $\mathbf X^\star=a\sum_{j=1}^r \bar \xi_j\mathbf u_j\mathbf u_j^T$ with $\xi_j\ge 0$, $j\in[r]$, one has $\mathbf X^\star\succeq 0$. 
From this and since $\mathcal{A}\mathbf X^\star=\mathbf b$, 
$\mathbf X^\star$ is a feasible solution of SDP \eqref{eq:SDP.form.0}. 
Moreover,
\[\begin{array}{rl}

    \left< \mathbf C, \mathbf X^\star\right>&=\left< \mathbf C-\mathcal{A}^T\mathbf{\bar z}, \mathbf X^\star\right>+\left< \mathcal{A}^T\mathbf{\bar z}, \mathbf X^\star\right>\\
    &=a\sum_{j=1}^r \bar \xi_j\left< \mathbf C-\mathcal{A}^T\mathbf{\bar z},\mathbf u_j\mathbf u_j^T\right>+\mathbf{\bar z}^T( \mathcal{A} \mathbf X^\star)\\
    &=a\sum_{j=1}^r \bar \xi_j\mathbf u_j^T( \mathbf C-\mathcal{A}^T\mathbf{\bar z})\mathbf u_j+\mathbf{\bar z}^T\mathbf b\\
    &=a\lambda_1(\mathbf C-\mathcal{A}^T\mathbf{\bar z})\sum_{j=1}^r \bar \xi_j\|\mathbf u_j\|_2^2+\mathbf{\bar z}^T\mathbf b\\
    &=a\lambda_1(\mathbf C-\mathcal{A}^T\mathbf{\bar z})\sum_{j=1}^r \bar \xi_j+\mathbf{\bar z}^T\mathbf b\\
    &=a\lambda_1(\mathbf C-\mathcal{A}^T\mathbf{\bar z})+\mathbf{\bar z}^T\mathbf b=\varphi(\mathbf{\bar z})=-\tau\,.
\end{array}\]
Thus, $\left< \mathbf C, \mathbf X^\star\right>=-\tau$, yielding the third statement.
\end{proof}
To obtain a convergence guarantee when solving NSOP \eqref{eq:nonsmooth.hierarchy.0} by LMBM \cite[Algorithm 1]{haarala2007globally}, we need the following technical lemma:
\begin{lemma}\label{lem:LMBM.conver.guara}
When applied to problem NSOP
  \eqref{eq:nonsmooth.hierarchy.0}, the LMBM
 algorithm is globally convergent.

\end{lemma}
\begin{proof}
The convexity of $\varphi$ yields that  $\varphi$ is weakly upper semismooth on $\R^m$ according to \cite[Proposition 5]{mifflin1977algorithm}. 
From this, $\varphi$ is upper semidifferentiable on $\R^m$  by using \cite[Theorem 3.1]{bihain1984optimization}.
Combining this with the fact that $\varphi$ is bounded from below on $\R^m$, the result follows thanks to  \cite[Section 5]{bihain1984optimization} (see also the final statement of \cite[Section 14.2]{bagirov2014introduction}).
\end{proof}

\subsubsection{SDP with Bounded Trace Property}
 \paragraph{Proof of Lemma \ref{lem:obtain.dual.sol.btp}:}
 \label{proof:lem:obtain.dual.sol.btp}
\begin{proof}
Let $\mathbf X^\star$ be an optimal solution of SDP \eqref{eq:SDP.form.0} and set $\ushort a:=\trace(\mathbf X^\star)$.
By Condition 5 of Assumption \ref{ass:general.assump.sdp}, one has
\begin{equation}\label{eq:ineq.a.bar.a}
    a \ge \ushort a> 0\,.
\end{equation}
Similarly to the proof of  Lemma \ref{lem:obtain.dual.sol}, one obtains:
\begin{equation}\label{eq:spectral.form.of.trace}
      -\tau=  \inf \{\ushort a\lambda_1(\mathbf C-\mathcal{A}^T \mathbf z)+ \mathbf b^T\mathbf z \,:\,\mathbf z\in\R^m \}\,. 
\end{equation}
Let us prove that 
\begin{equation}\label{eq:lower.boun.psi}
    \psi(\mathbf z)\ge -\tau\,,\,\forall \mathbf{z}\in \R^m\,.
\end{equation}
Let $\mathbf{z}\in \R^m$ be fixed and consider the following two cases:
\begin{itemize}
    \item Case 1: $\lambda_1(\mathbf C-\mathcal{A}^T\mathbf z)>0$. By \eqref{eq:ineq.a.bar.a} and \eqref{eq:spectral.form.of.trace}, 
    \[\psi(\mathbf z)= a \lambda_1(\mathbf C-\mathcal{A}^T\mathbf z) +\mathbf b^T \mathbf z\ge \ushort a\lambda_1(\mathbf C-\mathcal{A}^T \mathbf z)+ \mathbf b^T\mathbf z\ge -\tau\,. \]
    Thus, $\psi(\mathbf z)\ge -\tau$.
    \item Case 2: $\lambda_1(\mathbf C-\mathcal{A}^T\mathbf z)\le 0$. Then $\mathcal{A}^T\mathbf z-\mathbf C\succeq 0$ and $\psi(\mathbf z)=\mathbf b^T\mathbf z\ge -\rho=-\tau$ by \eqref{eq:SDP.form.dual.0}.
\end{itemize}
Let $(\mathbf z^{(j)})_{j\in\N}$ be a minimizing sequence of SDP \eqref{eq:SDP.form.dual.0}.
Then $\lambda_1(\mathbf C-\mathcal{A}^T\mathbf z^{(j)})\le 0$, $j\in\N$, since $\mathcal{A}^T\mathbf z^{(j)}-\mathbf C\succeq 0$ and $\mathbf b^T\mathbf z^{(j)}\to -\tau$ as $j\to \infty$ since $\tau=\rho$. 
It implies that $\psi(\mathbf z^{(j)})=\mathbf b^T\mathbf z^{(j)}\to -\tau$ as $j\to \infty$. 
From this and by \eqref{eq:lower.boun.psi}, the first statement follows.

For the second statement, let $\mathbf z^\star$ be an optimal solution of SDP \eqref{eq:SDP.form.dual.0}. 
Since $\mathcal{A}^T\mathbf z-\mathbf C\succeq 0$, $\lambda_1(\mathbf C-\mathcal{A}^T\mathbf z)\le 0$ and thus $\psi(\mathbf z^\star)=\mathbf b^T\mathbf z^\star= -\rho=-\tau$. 
Thus, $\mathbf z^\star$ is an optimal solution of \eqref{eq:nonsmooth.hierarchy.0.btp}, yielding the  second statement.
\end{proof}

We consider the differentiability properties of $\psi$ in the following proposition:
\begin{proposition}\label{prop:properties.phi.0.btp}
The function $\psi$ has the following properties:
\begin{enumerate}
\item  $\psi$ is convex and continuous but not differentiable.
\item The subdifferential of $\psi$ at $\mathbf z$ reads:
\begin{equation*}\label{eq:subgrad.phi.fo.0.btp}
    \partial \psi(\mathbf z)=\begin{cases}
    \{\mathbf b\}\text{ if }\lambda_1(\mathbf C-\mathcal{A}^T\mathbf z)< 0 \,,\\
    \{\mathbf b-a\mathcal{A}\mathbf W \ :\ \mathbf W\in \conv(\Gamma(\mathbf C-\mathcal{A}^T\mathbf z))\}\text{ if }\lambda_1(\mathbf C-\mathcal{A}^T\mathbf z)> 0\,,\\
    \{\mathbf b-\zeta a\mathcal{A}\mathbf W \ :\ \zeta\in [0,1]\,,\,\mathbf W\in \conv(\Gamma(\mathbf C-\mathcal{A}^T\mathbf z))\}\text{ otherwise}\,,\\
    \end{cases}
\end{equation*}
where $\Gamma(.)$ is defined as in \eqref{eq:Gamma}.
\end{enumerate}
\end{proposition}
\begin{proof}
Note that $\psi$ is the maximum of two convex functions, i.e., 
\[\psi(\mathbf z)=\max\{\varphi_1(\mathbf z),\varphi_2(\mathbf z)\} \,,\] 
with $\varphi_1(\mathbf z)=a\lambda_1(\mathbf C-\mathcal{A}^T\mathbf z)+\mathbf b^T\mathbf z$ and $\varphi_2(\mathbf z)=\mathbf b^T\mathbf z$. 
Thus, $\psi$ is convex and 
\[\partial \psi(\mathbf z)= \begin{cases}
\partial \varphi_1(\mathbf z) &\text{ if }\varphi_1(\mathbf z)>\varphi_2(\mathbf z)\,,\\
\conv(\partial \varphi_1(\mathbf z)\cup \partial\varphi_2(\mathbf z))&\text{ if }\varphi_1(\mathbf z)=\varphi_2(\mathbf z)\,,\\
\partial \varphi_2(\mathbf z) &\text{ otherwise}\,.
\end{cases}
\]
Note that $\partial \varphi_2(\mathbf z)=\{\mathbf b\}$ and $\partial \varphi_1(\mathbf z)$ is computed as in formula \eqref{eq:subgrad.phi.fo.0}.
Thus, the result follows.
\end{proof}
The following theorem is useful to recover an optimal solution of SDP \eqref{eq:SDP.form.0} from an optimal solution of NSOP \eqref{eq:nonsmooth.hierarchy.0.btp}.
\begin{lemma}\label{lem:extract.sdp.solu.btp}
Assume that $\mathbf{\bar z}$ is an optimal solution of NSOP \eqref{eq:nonsmooth.hierarchy.0.btp}. 
The following statements are true:
\begin{enumerate}
    \item There exists 
    \[\mathbf X^\star\begin{cases}
    = \mathbf 0& \text{ if } \lambda_1(\mathbf C-\mathcal{A}^T\mathbf{\bar z})< 0\,,\\
    \in\zeta a\conv(\Gamma(\mathbf C-\mathcal{A}^T\mathbf{\bar z})) &\text{ if } \lambda_1(\mathbf C-\mathcal{A}^T\mathbf{\bar z})=0\,,\\
    \in a\conv(\Gamma(\mathbf C-\mathcal{A}^T\mathbf{\bar z}))& \text{ otherwise}\,,
    \end{cases}\]
    for some $\zeta\in[0,1]$ such that $\mathcal{A}\mathbf X^\star=\mathbf b$.
    \item $\mathbf X^\star=\sum_{j=1}^r \bar \xi_j\mathbf u_j\mathbf u_j^T$ where $\mathbf u_1,\dots,\mathbf u_r$ are all uniform eigenvectors  corresponding to $\lambda_1(\mathbf C-\mathcal{A}^T\mathbf{\bar z})$ and $(\bar \xi_1,\dots,\bar \xi_r)$ is an optimal solution of QP \eqref{eq:socp.sdp.sol.btp}. 
    \item $\mathbf X^\star$ is an optimal solution of SDP \eqref{eq:SDP.form.0}.
\end{enumerate}
\end{lemma}
\begin{proof}
Due to \cite[Theorem 4.2]{bagirov2014introduction}, $\mathbf 0\in \partial \varphi(\mathbf{\bar z})$. From this and by Proposition \ref{prop:properties.phi.0.btp}.2, 
the first statement follows. The second statement is implied by the first one.
Let us prove the third statement.
Since $\mathbf X^\star=\sum_{j=1}^r \bar \xi_j\mathbf u_j\mathbf u_j^T$ with $\xi_j\ge 0$, $j\in [r]$, one has $\mathbf X^\star\succeq 0$. 
From this and since $\mathcal{A}\mathbf X^\star=\mathbf b$, 
$\mathbf X^\star$ is a feasible solution of SDP \eqref{eq:SDP.form.0}. 
Moreover,
\[\begin{array}{rl}

    \left< \mathbf C, \mathbf X^\star\right>&=\left< \mathbf C-\mathcal{A}^T\mathbf{\bar z}, \mathbf X^\star\right>+\left< \mathcal{A}^T\mathbf{\bar z}, \mathbf X^\star\right>\\
    &=\sum_{j=1}^r \bar \xi_j\left< \mathbf C-\mathcal{A}^T\mathbf{\bar z},\mathbf u_j\mathbf u_j^T\right>+\mathbf{\bar z}^T( \mathcal{A} \mathbf X^\star)\\
    &=\sum_{j=1}^r \bar \xi_j\mathbf u_j^T( \mathbf C-\mathcal{A}^T\mathbf{\bar z})\mathbf u_j+\mathbf{\bar z}^T\mathbf b\\
    &=\lambda_1(\mathbf C-\mathcal{A}^T\mathbf{\bar z})\sum_{j=1}^r \bar \xi_j\|\mathbf u_j\|_2^2+\mathbf{\bar z}^T\mathbf b\\
    &=\lambda_1(\mathbf C-\mathcal{A}^T\mathbf{\bar z})\sum_{j=1}^r \bar \xi_j+\mathbf{\bar z}^T\mathbf b\\
    &=a\max \{\lambda_1(\mathbf C-\mathcal{A}^T\mathbf{\bar z}),0\}+\mathbf{\bar z}^T\mathbf b=\psi(\mathbf{\bar z})=-\tau\,.
\end{array}\]
Thus, $\left< \mathbf C, \mathbf X^\star\right>=-\tau$, yielding the third statement.
\end{proof}
The next result proves that
when applied to NSOP \eqref{eq:nonsmooth.hierarchy.0.btp},
the LMBM algorithm \cite[Algorithm 1]{haarala2007globally} converges.

\begin{lemma}\label{lem:LMBM.conver.guara.btp}
LMBM applied to 
NSOP \eqref{eq:nonsmooth.hierarchy.0.btp} is globally convergent.
\end{lemma}
The proof of Lemma \ref{lem:LMBM.conver.guara.btp} is similar to Lemma \ref{lem:LMBM.conver.guara}.

\subsection{Converting moment relaxations to standard SDP}
\label{sec:convert.SDP.sphere}
We will present a way to transform SDP \eqref{eq:moment.hierarchy.multiply.diagonal} to the form \eqref{eq:SDP.form} recalled as follows:
\begin{equation*}
-\tau_k = \sup_{\mathbf X\in \mathcal{S}_k} \{ \left< \mathbf C,\mathbf X\right>\,:\,\left<\mathbf A_j, \mathbf X\right>=\mathbf b_j\,,\,j\in [m]\,,\, \mathbf X \succeq 0\}\,.
\end{equation*}
Let $k\in\N$ be fixed.
We will prove that there exists $\mathbf A_j\in \mathcal S_k$, $j\in [r]$, such that $\mathbf X=\mathbf P_k\mathbf M_k(\mathbf y)\mathbf P_k$ for some $\mathbf y\in \R^{\stirlingii [0.7]{n} {2k}}$ if and only if  $\left< \mathbf A_j, \mathbf X\right>=0$, $j\in[r]$.
Let $\mathcal{V}=\{\mathbf P_k\mathbf M_k(\mathbf z)\mathbf P_k\,:\, \mathbf z\in \R^{\stirlingii [0.7]{n} {2k}}\}$.  Then $\mathcal{V}$ is a linear subspace of $\mathcal S_k$ and $\dim(\mathcal{V})=\stirlingii {n} {2k}$. 
We take a basis $\mathbf A_1,\dots,\mathbf A_r$ of the orthogonal complement $\mathcal{V}^\bot$ of $\mathcal{V}$. 
Notice that 
\begin{equation}\label{eq:dim.comple}
r=\dim(\mathcal{V}^\bot)=\dim(\mathcal S_k)-\stirlingii {n} {2k}=\frac{\stirlingii {n}k(\stirlingii {n}k+1)}2-\stirlingii {n}{2k}\,.
\end{equation}
With $\mathbf X\in \mathcal S_k$, it implies that $\mathbf X\in \mathcal{V}$ if and only if $\left< \mathbf A_j, \mathbf X\right>=0$, $j\in [r]$.

Let us find such a basis $\mathbf A_1,\dots,\mathbf A_r$. 
Let $\mathbf A=(A_{\alpha,\beta})_{\alpha,\beta\in\N^n_k}\in \mathcal{V}^\bot$. Then for all $\mathbf X=(X_{\alpha,\beta})_{\alpha,\beta\in\N^n_k}\in \mathcal{V}$, $\left< \mathbf A, \mathbf X\right>=0$. 
Note that if $\mathbf X=\mathbf P_k\mathbf M_k(\mathbf y)\mathbf P_k$, then one has 
\begin{equation*}\label{eq:relation.X.y}
X_{\alpha,\beta}=w_{\alpha,\beta}y_{\alpha+\beta}\,,\,\forall \alpha,\beta\in \N^n_k\,, 
\end{equation*}
with $w_{\alpha,\beta}:=\theta_{k,\alpha}^{1/2}\theta_{k,\beta}^{1/2}$, for all $\alpha,\beta\in \N^n_k$.
It implies that
\[0=\sum_{\alpha,\beta\in\N^n_k}w_{\alpha,\beta}A_{\alpha,\beta}y_{\alpha+\beta}\,,\,\forall \mathbf y\in\R^{\stirlingii[0.7] {n}{2k}}\,.\]
Let $\gamma\in\N^n_{2k}$ be fixed and let $\mathbf y\in\R^{\stirlingii [0.7]{n}{2k}}$ be such that for $\xi\in\N^n_{2k}$,
\[y_\xi=\begin{cases}
0&\text{ if }\xi\ne \gamma\,,\\
1&\text{ otherwise.}
\end{cases}\]
Then 
\[\arraycolsep=1.4pt\def\arraystretch{.5}
0=\sum_{\begin{array}{cc}
\scriptstyle \alpha,\beta\in\N^n_k \\
 \scriptstyle  \alpha+\beta=\gamma
\end{array}}
w_{\alpha,\beta}A_{\alpha,\beta}=\sum_{\begin{array}{cc}
\scriptstyle  \alpha,\beta\in\N^n_k \\
\scriptstyle \alpha=\beta=\gamma/2 \end{array}}w_{\alpha,\beta}A_{\alpha,\beta}+2\sum_{\begin{array}{cc}
\scriptstyle \alpha,\beta\in\N^n_k\\
\scriptstyle \alpha+\beta=\gamma\\
\scriptstyle \alpha<\beta
\end{array}
}w_{\alpha,\beta}A_{\alpha,\beta}\,.\]
If $\gamma\not\in 2\N^n$, we do not have the first term in the latter equality.
Let us define 
\[\Lambda_\gamma:=\{A_{\alpha,\beta}\,:\,\alpha,\beta\in\N^n_k\,,\, \alpha+\beta=\gamma\,,\, \alpha\le \beta\}\,.\]
It can be rewritten as $\Lambda_\gamma=\{A_{\alpha_j,\beta_j}\,,\,j\in [t]\}$
where $(\alpha_1,\beta_1)<\dots<(\alpha_t,\beta_t)$ and $t=|\Lambda_\gamma|$.
Thus, if $t\ge 2$, we can choose $\mathbf A$ such that for all $\alpha,\beta\in\N^n_k$,
\[A_{\alpha,\beta}=\begin{cases}
w_{\alpha_\mu,\beta_\mu}& \text{ if } \alpha_1=\beta_1 \text{ and }(\alpha,\beta) =(\alpha_1,\beta_1)\,,\\
\frac{1}2 w_{\alpha_\mu,\beta_\mu}& \text{ if } \alpha_1< \beta_1 \text{ and }(\alpha,\beta) \in\{(\alpha_1,\beta_1),(\beta_1,\alpha_1)\}\,,\\
-w_{\alpha_1,\beta_1}& \text{ if } \alpha_\mu=\beta_\mu \text{ and }(\alpha,\beta) =(\alpha_\mu,\beta_\mu)\,,\\
-\frac{1}2 w_{\alpha_1,\beta_1}& \text{ if } \alpha_\mu< \beta_\mu \text{ and }(\alpha,\beta) \in\{(\alpha_\mu,\beta_\mu),(\beta_\mu,\alpha_\mu)\}\,,\\
0 & \text{ otherwise}\,, 
\end{cases}\]
for some $\mu\in[t]\backslash \{1\}$.
Let us denote by $\mathcal{B}_\gamma$ the set of all such $A$ above satisfying $t=|\Lambda_\gamma|\ge 2$ and let $\mathcal{B}_\gamma=\emptyset$ otherwise. 
Then $|\mathcal{B}_\gamma|=|\Lambda_\gamma|-1 $.
From this and since $( \mathcal{B}_\gamma)_{\gamma\in\N^{n}_{2k}}$ is a sequence of pairwise disjoint subsets of $\mathcal S_k$, 
\[\left|\bigcup_{\gamma\in\N^{n}_{2k}} \mathcal{B}_\gamma\right|=\sum_{\gamma\in\N^{n}_{2k}}|\mathcal{B}_\gamma|=\sum_{\gamma\in\N^{n}_{2k}} \left( |(\alpha,\beta)\in(\N^n_k)^2\,:\, \alpha+\beta=\gamma\,,\,\alpha\le \beta | -\stirlingii {n}{2k} \right) \,.\]
It must be equal to $r$ as in \eqref{eq:dim.comple}. 
We just proved that $\bigcup_{\gamma\in\N^{n}_{2k}} \mathcal{B}_\gamma$ is a basis of $\mathcal{V}^\bot$.
Now we assume that $\bigcup_{\gamma\in\N^{n}_{2k}} \mathcal{B}_\gamma=\{\mathbf A_1,\dots,\mathbf A_r\}$.

Let us rewrite the constraints 
\begin{equation}\label{eq:moment.eq.cons}
\mathbf M_{k - \lceil h_j \rceil }(h_j\;\mathbf y)   = 0\,,\, j\in [l_h]\,,
\end{equation}
 as $\left<\mathbf A_j,\mathbf X\right>=0$, $j=r+1,\dots,m-1$
with $\mathbf X=\mathbf P_k\mathbf M_k(\mathbf y)\mathbf P_k$. 
From \eqref{eq:moment.eq.cons},
 \begin{equation}\label{eq:eq.hy}
\sum_{\gamma\in\N^n_{2\lceil h_j \rceil}}h_{j,\gamma}y_{\alpha+\gamma}=0\,,\,\alpha\in\N^n_{2(k-\lceil h_j \rceil)}\,,\,j\in [l_h]\,.
\end{equation}
Let $j\in[l_h]$ and $\alpha\in\N^n_{2(k-\lceil h_j \rceil)}$  be fixed. 
We define $\tilde {\mathbf A}=(\tilde A_{\mu,\nu})_{\mu,\nu\in\N^n_k}$ as follows:
\begin{equation}\label{eq:A.tilde}
    \tilde A_{\mu,\nu}=\begin{cases}
h_{j,\gamma}&\text{ if }\mu=\nu\,,\,\mu+\nu=\alpha+\gamma\,,\\

\frac{1}2 h_{j,\gamma}&\text{ if }\mu\ne\nu\,,\,\mu+\nu=\alpha+\gamma\,,\\
&\qquad\text{ and } (\mu,\nu)\le (\bar\mu,\bar\nu)\,,\,\forall \bar\mu,\bar\nu\in\N^n_k\, \text{ such that} \,\bar\mu+\bar\nu=\alpha+\gamma\,,\\
0&\text{ otherwise.}
\end{cases}
\end{equation}

Then \eqref{eq:eq.hy} implies that $\left<\tilde {\mathbf A}, \mathbf M_k(\mathbf y)\right>=0$.
Since $\mathbf M_k(\mathbf y)=\mathbf P_k^{-1}\mathbf X\mathbf P_k^{-1}$, 
\[0=\left<\tilde {\mathbf A}, \mathbf P_k^{-1}\mathbf X\mathbf P_k^{-1}\right>=\left<\mathbf P_k^{-1}\tilde {\mathbf A}\mathbf P_k^{-1}, \mathbf X\right>=\left<\mathbf A, \mathbf X\right>\,,\]
where $\mathbf A:=\mathbf P_k^{-1}\tilde{\mathbf  A}\mathbf P_k^{-1}$, yielding the statement. Thus, we obtain the constraints $\left<\mathbf A_j,\mathbf X\right>=0$, $j\in [m-1]$.

The final constraint $y_0=1$ can be rewritten as $\left<\mathbf A_m, \mathbf X\right>=1$ with $\mathbf A_m\in \mathcal S_k$ having zero entries  except the top left one $[\mathbf A_{m}]_{0,0}=1$.
Thus, we select $\mathbf b$ such that all entries of $\mathbf b$ are zeros except $b_m=1$.

The number $m$ (or $m_k$ when plugging the relaxation order $k$) of equality trace constraints $\left<\mathbf A_j,\mathbf X\right>=b_j$ is:
\begin{equation}\label{eq:bound.mk}
    m = \frac{1}2 {\stirlingii {n}k(\stirlingii {n}k+1)}-\stirlingii {n}{2k}+1+\sum_{j=1}^{l_h} {\stirlingii {n}{2(k-\lceil h_j \rceil)}} \,.
\end{equation}

The function $-L_{\mathbf y}(f)=-\sum_\gamma f_\gamma y_{\gamma}$ is equal to $\left<\mathbf C, \mathbf X\right>$ with $\mathbf C:=\mathbf P_k^{-1}\tilde {\mathbf C}\mathbf P_k^{-1}$, where $\tilde {\mathbf C}=(\tilde C_{\mu,\nu})_{\mu,\nu\in\N^n_k}$ is defined by:
\begin{equation}\label{eq:C.tilde}
    \tilde C_{\mu,\nu}=\begin{cases}
-f_{\gamma}&\text{ if }\mu=\nu\,,\,\mu+\nu=\gamma\,,\\

-\frac{1}2 f_{\gamma}&\text{ if }\mu\ne\nu\,,\,\mu+\nu=\gamma\,,\\
&\qquad\text{ and } (\mu,\nu)\le (\bar\mu,\bar\nu)\,,\,\forall \bar\mu,\bar\nu\in\N^n_k\, \text{ such that}\,\bar\mu+\bar\nu=\gamma\,,\\
0&\text{ otherwise.}
\end{cases}
\end{equation}


%
\footnotesize
\bibliographystyle{abbrv}

\end{document}